\documentclass[11pt,reqno]{amsart}

\usepackage{amssymb}
\usepackage{amsmath}
\usepackage{amsfonts} 
\usepackage{amsbsy}
\usepackage{amscd}
\usepackage{amsthm}
\usepackage{geometry}    
\usepackage{courier}
\usepackage{xcolor}

\usepackage{graphicx}
\usepackage{amssymb}
\usepackage{epstopdf}
\usepackage{mathtools}
\usepackage{mathrsfs}
\usepackage[capitalize]{cleveref}
\usepackage{mathrsfs} 
\usepackage[framemethod=tikz]{mdframed} 
\usepackage{circuitikz}
\usetikzlibrary{decorations.pathmorphing, patterns,shapes}

\usetikzlibrary{arrows}
\usetikzlibrary{patterns}
\usepackage{xcolor}
\usepackage{caption}
\usepackage{subcaption}

\newcommand{\Z}{\mathbb{Z}}

\newcommand{\Hidden}[1]{}

\newtheorem{theorem}{Theorem}

\newtheorem{prop}[theorem]{Proposition}
\newtheorem{proposition}[theorem]{Proposition}

\newtheorem{conjecture}[theorem]{Conjecture}
\newtheorem{op}[theorem]{Open Problem}
\theoremstyle{definition}

\newtheorem{example}[theorem]{Example}
\theoremstyle{remark}

\title[Recursions Satisfied by Families of Determinants]{Recursions Satisfied by Families of Determinants with Applications to Resistance Distance}
\author{Emily J. Evans}
\address{Brigham Young University}
\email{EJEvans@math.byu.edu}
\author{Russell Jay Hendel}
\address{Towson University}
\email{RHendel@Towson.Edu}

\begin{document}

\begin{abstract} The main contribution of this paper is a six-step semi-automatic algorithm that obtains a recursion satisfied by a family of determinants by systematically and iteratively applying Laplace expansion to the underlying matrix family. The recursion allows explicit computation of the Binet form providing a closed formula for resistance distance between two specified nodes in a family of graphs. This approach is particularly suited for graph families with complex structures; the method is used to prove the 1 over 14 conjectured asymptotic formula for linear 3--trees. Additionally, although the literature on recursive formulas for resistance distances is quite  large, the Fibonacci Quarterly and the Proceedings have almost no such results despite the fact that many recursions related to resistance distances involve the Fibonacci numbers. Therefore, a secondary purpose of the paper is to provide a brief  introductory survey of graph families, accompanied by figures, Laplacian matrices, and typical recursive results, supported by a modest bibliography of current papers on many relevant graph families, in the hope to involve Fibonaccians in this active and beautiful field.  
   
\end{abstract}

\maketitle


\section{Introduction}
Resistance distance, also referred to as effective resistance, is a well-known metric on graphs that measures both the number of paths between two vertices in a graph and the cost of each path.  A wide variety of applications of resistance distance exist including applications to mathematical chemistry~\cite{rdmatrix,carmona2014effective,Cinkir,KleinRandic, klein2002resistance,klein2004random, kem1,peng2017kirchhoff, yang2014comparison, yang2008kirchhoff}, graph theory~\cite{bapatdvi,BapatWheels,MarkK, DEVRIENDT202224,littleswim, Ghosh, klein1997graph, ZHOU20172864}, numerical linear algebra~\cite{SpielSparse}, and engineering~\cite{Barooah06grapheffective}.  Of special interest to graph theorists and mathematical chemists is the  calculation of resistance distance in families of graphs.  We say a graph is a member of a family if a particular structure is maintained as the number of nodes grows.

Given a graph $G$ the resistance distance between two nodes is determined by considering the graph as an electric circuit where each edge is represented by a resistor whose resistance is the inverse of the edge weight. Given any two nodes $i$ and $j$ assume that one unit of current flows into node $i$ and one unit of current flows out of node $j$.  The potential difference $v_i - v_j$ between nodes $i$ and $j$ needed to maintain this current is the {\it resistance distance} between $i$ and $j$. 

There are many particular families of graphs for which general resistance formulae have been obtained such as polyacene~\cite{polyacene}, fullerene~\cite{fullerene}, circulant~\cite{circulant}, corona~\cite{corona}, octogonal~\cite{octogonal}, regular~\cite{regular}, (almost) complete bipartite~\cite{bipartite}, Cayley \cite{cayley2, cayley}, cubic~\cite{cubic}, ring clique~\cite{ringclique}, straight linear 1 and 2 tree~ \cite{bef}, Apollonian~\cite{apollonian}, flower \cite{flower}, 	Sierpinski~\cite{sierpinski}, and ladder~\cite{Cinkir} graphs and network families.

To obtain these formulae, many methods have been utilized (for a summary, including worked examples see~\cite{littleswim}) but the most common include the use of matrices (the combinatorial Laplacian), circuit theory transformations that preserve resistance distance, and  graph theoretic approaches.

Among methods that use the Laplacian matrix, one technique uses determinants associated with the underlying  matrix with specific rows and columns deleted~\cite{bapatdvi}.  In this paper, we approach the computation of a determinant by calculating the recursion satisfied by the underlying family. These recursions allow us to compute Binet forms and, as a consequence, compute closed-formula for resistances. The use of recursive relationships satisfied  by families of determinants is not a new idea; many well known formulas exist
(e.g., tridiagonal matrices~\cite{ELMIKKAWY2004669}, pentadiagonal matrices~\cite{pentadiag,recursion}
block tridiagonal matrices~\cite{molinari2008determinants}, and Toeplitz matrices~\cite{li2011calculating}).  By performing  Laplace expansion  to calculate the determinants, \cite{pentadiag} shows that the determinants of the general pentadiagonal family of matrices governed by five parameters, satisfies a sixth order recursion whose roots can be explicitly calculated.  

Building on the idea presented in \cite{pentadiag}  of calculating the resistance by first computing the closed recursive formula satisfied by the determinants of the underlying graph family,  the major contribution of this paper (Section \ref{sec:algorithm}) is a six-step semi-automatic algorithm which can be used to derive closed formulae for the resistance distance in graph families.   This algorithm is supported by a collection of  lemmas assistive in the calculations. Additionally, several conjectures connected with this method of Laplace expansions are presented. Software programs, written in Mathematica 13.3 supporting this semi-automatic algorithm are presented in Appendix I.

The key strength of this approach is its ability to automate very complicated computations. Using  this method we are able to prove a conjecture about asymptotic resistance distances in linear 3--trees, which hitherto could not be proven.

While, as just indicated, the literature on resistance distance in graph families is quite large and these formula frequently involve Fibonacci numbers, there are [almost] no papers in the Fibonacci Quarterly or the Proceedings discussing resistance distance. More precisely, except for the recent paper by the authors \cite{Sarajevo}, there are only  five papers that were published between 25 and 40 years ago, three of which deal with ladder networks \cite{Ferri, Lahr, Risk} and one of which deals with a static, carry, look-ahead gate \cite{Nodine}.

 Consequently, a secondary goal of this paper is to present a light, brief, introductory survey of graph families and their associated recursions in the hope to interest Fibonaccians in pursuing this beautiful and active field.
 Additionally, the modest bibliography of current papers  on many relevant graph families reviewed in the preceding paragraphs, should also be useful.
Towards this end, Section \ref{sec:examples} presents half a dozen graph families, accompanied by graphs, their combinatorial Laplacian, and sample recursive formula connected with them.

\section{Sample families of graphs and notational conventions}\label{sec:examples}
This section  introduces several examples of families of graphs. For each graph we provide a definition,  a figure, a representative Laplacian matrix, and known or conjectured results. As indicated in the introductory section, these graphs furnish a survey of basic examples of graphs whose resistance distance has been calculated and allow illustration of the Laplace expansion approach presented in this paper. The examples were chosen because they met one or more of the following criteria:
\begin{enumerate}
    \item proof of an unproven conjecture,
    \item illustration of simple applications of the method,
    \item illustrations of complicated applications of the method as an alternative to long proofs using many lemmas, or
    \item  illustration of challenges associated with the method.  
\end{enumerate}

\subsection{Path graphs}\label{subsec:path}   The simplest family of graphs under consideration is the family of path graphs as shown in Figure~\ref{fig:path}.  We recall that the graph Laplacian is defined as $L=D-A$ where $D$ is a diagonal matrix with the degree of each vertex as its entries, and $A$ is the graph adjacency matrix.   The adjacency matrix for the path graph is the banded matrix with ones on both the super and sub diagonal, hence the Laplacian is given by
\[L_G=\left[\begin{array}{rrrrrrrr}
1 & -1 & 0 & 0  &   0& \dots & \dots &0\\
-1 & 2 & - 1 & 0 &    0&  0& \ddots& \vdots\\
0 &-1 & 2 & - 1 & 0 &    0&  \ddots&  \vdots\\
0 &0 &-1 & 2 & - 1 & 0 &  \ddots&  0\\
0 & \ddots &\ddots & \ddots &  \ddots &  \ddots&  \ddots & 0\\
\vdots & \ddots & 0 & 0 &  -1 &  2&  -1 & 0\\
\vdots & \ddots  &0 & 0 & 0 &  -1 &  2&  -1 \\[1.7mm]
0  &\dots& \dots & 0 &0& 0 &  -1 & 1
\end{array}\right].\]

	\begin{figure}
\begin{center}

\begin{tikzpicture}[line cap=round,line join=round,>=triangle 45,x=1.0cm,y=1.0cm, scale = 1.2]
\draw [line width=.8pt] (-3.,2.)-- (-2.,2.);
\draw [line width=.8pt] (-2.,2.)-- (-1.,2.);
\draw [line width=.8pt,dotted] (-1.,2.)-- (0.,2.);
\draw [line width=.8pt] (0.,2.)-- (1.,2.);
\begin{scriptsize}
\draw [fill=black] (-3.,2.) circle (1.5pt);
\draw[color=black] (-3.0357596415464108,2.283574198210263) node {$1$};
\draw [fill=black] (-2.,2.) circle (1.5pt);
\draw[color=black] (-2.0176831175804275,2.269627944457304) node {$2$};
\draw [fill=black] (-1.,2.) circle (1.5pt);
\draw[color=black] (-1.0274991011203625,2.283574198210263) node {$3$};
\draw [fill=black] (0.,2.) circle (1.5pt);
\draw[color=black] (-0.037315084660296774,2.269627944457304) node {$n-1$};
\draw [fill=black] (1.,2.) circle (1.5pt);
\draw[color=black] (0.9528689317997687,2.297520451963222) node {$n$};
\end{scriptsize}
\end{tikzpicture}

%
%
\end{center}
\caption{A path graph on $n$ vertices.}
\label{fig:path}
\end{figure}
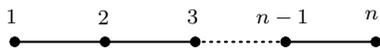

The resistance distance between two vertices in the path graph is the same as the distance, that is, for a path graph as shown in Figure~\ref{fig:path} the resistance distance between node $1$ and node $n$ is $n-1$.

\subsection{Notation}
\label{sub:notation}
Before continuing to introduce additional families of graphs, we pause to introduce some necessary notation and formulas used throughout the paper. If $M$ is an arbitrary matrix family corresponding to a family of graphs then we let $M^{n \; \times \; n}_{i,j}$ refer to the entry in row $i$ column $j$ of the $n \times n$ matrix in this family. If $A$ and $B$ are sets of indices (or singleton indices), then  $M^{n\; \times \; m}(A|B)$ is the matrix obtained from $M^{n \; \times \; n}$ by deleting the rows whose indices are described by $A$ and  deleting the columns whose indices are described by $B.$  If the sets are singletons we we will not use braces so that e.g., $L(1|n) $ means $L(\{1\}|\{n\}).$ Finally,  $Det(M^{n \; \times \;n})$ indicates the determinant of $M^{n \;\times\;n}.$ 

If $u$ and $v$ are two nodes in a family of graphs whose underlying Laplacian family of matrices is given by $L$ then the formula below gives us one method of computing the effective  resistance between nodes $u$ and $v.$ Throughout the paper, we have represented graph families  so that $u$ and $v$ correspond to nodes 1 and $n$ of the $n$-th member of the underlying graph family.  Using our notation we have the following formula due to Bapat \cite{Bapatbook}.
\begin{equation}\label{equ:Bapat}
        \text{Resistance distance between nodes 1 and $n$} = \frac{Det(L^n(\{1,n\}|\{1,n\}))}{Det(L^n(1|1))}.
\end{equation}
Equation \eqref{equ:Bapat} is also valid if the denominator is $Det(L^n(n|n)).$

\subsection{Straight Linear 2--trees}\label{sub:2tree}
The next family of graphs 
 under consideration, a  generalization of the path graph is the so-called straight linear two tree, sometimes referred to as a two-path in the literature.  This graph is shown in Figure \ref{fig:2tree}, and the $n \times n$ Laplacian matrix is:
\[L_G=\left[\begin{array}{rrrrrrrr}
\label{equ:2tree}
2 & -1 & -1 & 0  &   0& \dots & \dots &0\\
-1 & 3 & - 1 & -1 &    0&  0& \ddots& \vdots\\
-1 &-1 & 4 & - 1 & -1 &    0&  \ddots&  \vdots\\
0 &-1 &-1 & 4 & - 1 & -1 &  \ddots&  0\\
0 & \ddots &\ddots & \ddots &  \ddots &  \ddots&  \ddots & 0\\
\vdots & \ddots & 0 & -1 &  -1 &  4&  -1 & -1\\
\vdots & \ddots  &0 & 0 & -1 &  -1 &  3&  -1 \\[1.7mm]
0  &\dots& \dots & 0 &0& -1 &  -1 & 2
\end{array}\right].\]

\begin{figure}[h!]
\begin{center}

\begin{tikzpicture}[line cap=round,line join=round,>=triangle 45,x=1.0cm,y=1.0cm,scale = 1.2]
\draw [line width=1.pt] (-3.,0.)-- (-2.,0.);
\draw [line width=1.pt] (-2.,0.)-- (-1.,0.);
\draw [line width=1.pt,dotted] (-1.,0.)-- (0.,0.);
\draw [line width=1.pt] (0.,0.)-- (1.,0.);
\draw [line width=1.pt] (1.,0.)-- (2.,0.);
\draw [line width=1.pt] (2.,0.)-- (1.5,0.866025403784435);
\draw [line width=1.pt] (1.5,0.866025403784435)-- (1.,0.);
\draw [line width=1.pt] (1.,0.)-- (0.5,0.8660254037844366);
\draw [line width=1.pt] (0.5,0.8660254037844366)-- (0.,0.);
\draw [line width=1.pt] (-0.5,0.8660254037844378)-- (-1.,0.);
\draw [line width=1.pt] (-1.,0.)-- (-1.5,0.8660254037844385);
\draw [line width=1.pt] (-1.5,0.8660254037844385)-- (-2.,0.);
\draw [line width=1.pt] (-2.,0.)-- (-2.5,0.8660254037844388);
\draw [line width=1.pt] (-2.5,0.8660254037844388)-- (-3.,0.);
\draw [line width=1.pt] (-2.5,0.8660254037844388)-- (-1.5,0.8660254037844385);
\draw [line width=1.pt] (-1.5,0.8660254037844385)-- (-0.5,0.8660254037844378);
\draw [line width=1.pt,dotted] (-0.5,0.8660254037844378)-- (0.5,0.8660254037844366);
\draw [line width=1.pt] (0.5,0.8660254037844366)-- (1.5,0.866025403784435);
\begin{scriptsize}
\draw [fill=black] (-3.,0.) circle (1.5pt);
\draw[color=black] (-3.02279181666165,-0.22431183338253265) node {$1$};
\draw [fill=black] (-2.,0.) circle (1.5pt);
\draw[color=black] (-2.0001954862580344,-0.22395896857501957) node {$3$};
\draw [fill=black] (-2.5,0.8660254037844388) circle (1.5pt);
\draw[color=black] (-2.5018465162673555,1.100526386601008717) node {$2$};
\draw [fill=black] (-1.5,0.8660254037844385) circle (1.5pt);
\draw[color=black] (-1.5081915914412003,1.100526386601008717) node {$4$};
\draw [fill=black] (-1.,0.) circle (1.5pt);
\draw[color=black] (-1.0065405614318794,-0.22290037415248035) node {$5$};
\draw [fill=black] (-0.5,0.8660254037844378) circle (1.5pt);
\draw[color=black] (-0.4952423962300715,1.100526386601008717) node {$6$};
\draw [fill=black] (0.,0.) circle (1.5pt);
\draw[color=black] (-0.03217990699069834,-0.22431183338253265) node {$n-4$};
\draw [fill=black] (0.5,0.8660254037844366) circle (1.5pt);
\draw[color=black] (0.4887653934035965,1.103344389715898) node {$n-3$};
\draw [fill=black] (1.,0.) circle (1.5pt);
\draw[color=black] (0.9904164234129174,-0.22431183338253265) node {$n-2$};
\draw [fill=black] (1.5,0.866025403784435) circle (1.5pt);
\draw[color=black] (1.5692445349621338,1.1042991524908385) node {$n-1$};
\draw [fill=black] (2.,0.) circle (1.5pt);
\draw[color=black] (1.993718483431559,-0.22431183338253265) node {$n$};
\end{scriptsize}
\end{tikzpicture}
\end{center}
\caption{A straight linear 2-tree}
\label{fig:2tree}
\end{figure}
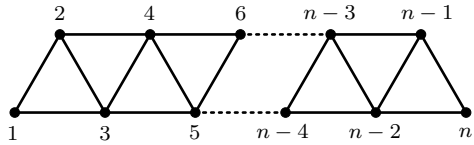

The resistance distance between any two vertices of this graph is known (see~\cite{bef}) but we only give the result between the vertices of degree two.
	\begin{theorem}\label{the:2tree}
 
	Let $G_n$ be the linear 2--tree with $n$ vertices. Then the resistance distance between nodes $1$ and $n$ is given by
	\[r(1,n)=\frac{2F_{n-1}^2}{L_{n-1}L_{n-2}}+\sum_{i=1}^{n-3}\frac{F_iF_{i+1}}{L_iL_{i+1}}
	 = \frac{n-1}{5} + \frac{4F_{n-1}}{5L_{n-1}},\]
  where $F_k$ and $L_k$ refer to the $k$th Fibonacci and Lucas numbers respectively.
	\end{theorem}

\subsection{Straight linear 3--trees}\label{sub:linear3tree}

We can further generalize the path graph and 2-tree by considering the \textit{straight linear 3--tree} (sometimes referred to in the literature as a 3--path) which is 3--tree where there are only two vertices of degree three, and whose adjacency matrix consists of ones on the first three super and sub diagonals and zero elsewhere.  An example of such a tree on six vertices is shown in Figure \ref{fig:st3tree}.
\begin{figure}[ht!]
    \centering
\begin{tikzpicture}[line cap=round,line join=round,>=triangle 45,x=1.0cm,y=1.0cm, scale = .8]
\draw [line width=.8pt] (0.7119648044207935,0.056502145327751115)-- (-0.19048706476012342,2.0870188509848124);
\draw [line width=.8pt] (-0.19048706476012342,2.0870188509848124)-- (2.6636787078204596,1.3307124522628444);
\draw [line width=.8pt] (2.6636787078204596,1.3307124522628444)-- (0.7119648044207935,0.056502145327751115);
\draw [line width=.8pt] (0.7119648044207935,0.056502145327751115)-- (1.2007929002271236,2.481841543751463);
\draw [line width=.8pt] (1.2007929002271236,2.481841543751463)-- (-0.19048706476012342,2.0870188509848124);
\draw [line width=.8pt] (1.2007929002271236,2.481841543751463)-- (2.6636787078204596,1.3307124522628444);
\draw [line width=.8pt] (1.5392123511699674,3.685110702659351)-- (-0.19048706476012342,2.0870188509848124);
\draw [line width=.8pt] (1.5392123511699674,3.685110702659351)-- (1.2007929002271236,2.481841543751463);
\draw [line width=.8pt] (1.5392123511699674,3.685110702659351)-- (2.6636787078204596,1.3307124522628444);
\draw [line width=.8pt] (1.5392123511699674,3.685110702659351)-- (3.43812149257148,3.478298815972058);
\draw [line width=.8pt] (3.43812149257148,3.478298815972058)-- (2.6636787078204596,1.3307124522628444);
\draw [line width=.8pt] (1.2007929002271236,2.481841543751463)-- (3.43812149257148,3.478298815972058);
\begin{scriptsize}
\draw [fill=black] (0.7119648044207935,0.056502145327751115) circle (2.pt);
\draw[color=black] (0.6743626432049219,-0.19731244287938157) node {1};
\draw [fill=black] (-0.19048706476012342,2.0870188509848124) circle (2.pt);
\draw[color=black] (-0.547707596310903,2.171623713720523) node {2};
\draw [fill=black] (2.6636787078204596,1.3307124522628444) circle (2.pt);
\draw[color=black] (2.836486913117535,1.2691718445396072) node {$3$};
\draw [fill=black] (1.2007929002271236,2.481841543751463) circle (2.pt);
\draw[color=black] (0.9563788523239585,2.773258293174467) node {$4$};
\draw [fill=black] (1.5392123511699674,3.685110702659351) circle (2.pt);
\draw[color=black] (1.3700026256985456,3.995328532690291) node {$5$};
\draw [fill=black] (3.43812149257148,3.478298815972058) circle (2.pt);
\draw[color=black] (3.6449333792587733,3.6757101623553834) node {$6$};
\draw [fill=black] (2.6636787078204605,1.3307124522628435) circle (2.0pt);
\end{scriptsize}

\end{tikzpicture}

\caption{A straight linear 3--tree on $6$ vertices.  We use the numbering convention show above, and in particular vertex 1 and vertex $n$ have degree 3.}
\label{fig:st3tree}
\end{figure}
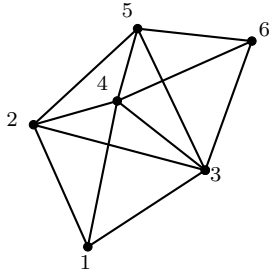

The Laplacian matrix, $L,$ of the straight linear 3--tree is given by

\begin{equation}\label{equ:PLinear3}
L_G=\begin{bmatrix}	
3 & -1 & -1 & -1 &\dotsc & 0 &  0 & 0\\
-1 & 4 & -1 & -1 & \dotsc &0 &  0 & 0 \\ 
-1 & -1 & 5 & -1 & \dotsc & 0 & 0 & 0\\
-1 & -1 & -1 & 6 & \dotsc & 0 & 0 & 0  \\
\vdots & \ddots & \ddots & \ddots &\ddots & 
\ddots & \ddots &  \vdots \\
0 & 0 & 0 & 0 & \dotsc & 5 & -1 & -1  \\
0 & 0 & 0 & 0 & \dotsc & -1 & 4 & -1 \\
0 & 0 & 0 & 0 & \dotsc & -1 & -1 & 3\\
\end{bmatrix}.
\end{equation}
The middle of the matrix continues with sixes on the diagonal and negative ones on the first three super and subdiagonals.  Although a formula for the resistance distance between vertices in the straight linear three tree is not known, the following limit has been conjectured \cite{bef} and will be  shown in this paper.
\begin{conjecture}\label{con:onefourteenth} Let $G$ be the straight linear 3--tree,  with $n$ vertices and $H$ be the straight linear 3--tree with $n+1$ vertices. Let $r_G(1,n)$ ($r_H(1,n-1)$) indicate the total resistance between the two corner nodes 1 and $n$ ($1$ and $n+1)$. 
	Then 
	\[\lim_{n\rightarrow \infty} r_{H} (1, n+1) - r_G(1,n) = \frac{1}{14}.\]
 \end{conjecture}

\subsection{Ladder graphs}\label{sub:ladder} An alternative generalization of the path graph is the so-called ladder graphs on $n=2m$ vertices as illustrated in Figure~\ref{fig:laddergraph}. This   graph is the Cartesian product of $P_m$ and $P_2$. The first known resistance distance results were obtained by Cinkir~\cite{Cinkir}. 

\begin{figure}[ht!]
\begin{center}
\begin{tikzpicture}[line cap=round,line join=round,>=triangle 45,x=1.0cm,y=1.0cm,scale = 1]
\draw [line width=.8pt] (3.,4.)-- (3.,3.);
\draw [line width=.8pt] (3.,4.)-- (4.,4.);
\draw [line width=.8pt] (3.,3.)-- (4.,3.);
\draw [line width=.8pt] (4.,4.)-- (4.,3.);
\draw [line width=.8pt] (4.,4.)-- (5.,4.);
\draw [line width=.8pt] (4.,3.)-- (5.,3.);
\draw [line width=.8pt] (5.,4.)-- (5.,3.);
\draw [fill=black] (5.3,4.) circle (.6pt);
\draw [fill=black] (5.5,4.) circle (.6pt);
\draw [fill=black] (5.7,4.) circle (.6pt);
\draw [fill=black] (5.3,3.) circle (.6pt);
\draw [fill=black] (5.5,3.) circle (.6pt);
\draw [fill=black] (5.7,3.) circle (.6pt);
\draw [line width=.8pt] (6.,4.)-- (6.,3.);
\draw [line width=.8pt] (6.,4.)-- (7.,4.);
\draw [line width=.8pt] (6.,3.)-- (7.,3.);
\draw [line width=.8pt] (7.,4.)-- (7.,3.);
\begin{small}
\draw [fill=black] (3.,4.) circle (1.6pt);
\draw[color=black] (3,4.2) node {$1$};
\draw [fill=black] (3.,3.) circle (1.6pt);
\draw[color=black] (3,2.8) node {$2$};
\draw [fill=black] (4.,4.) circle (1.6pt);
\draw[color=black] (4,4.2) node {$3$};
\draw [fill=black] (4.,3.) circle (1.6pt);
\draw[color=black] (4,2.8) node {$4$};
\draw [fill=black] (5.,4.) circle (1.6pt);
\draw [fill=black] (5.,3.) circle (1.6pt);
\draw [fill=black] (6.,4.) circle (1.6pt);
\draw [fill=black] (6.,3.) circle (1.6pt);
\draw [fill=black] (7.,4.) circle (1.6pt);
\draw[color=black] (7,4.2) node {$2m-1$};
\draw [fill=black] (7.,3.) circle (1.6pt);
\end{small}
\begin{scriptsize}

\draw[color=black] (7,2.8) node {$2m$};

\end{scriptsize}
\end{tikzpicture}

\end{center}
\caption{The ladder graph on $n=2m$ vertices.}\label{fig:laddergraph}
\end{figure}
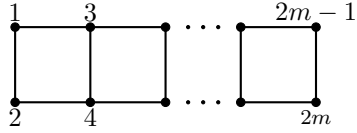

In particular we have the following theorem:
\begin{theorem}

	Let $G_n$ be the ladder graph with $n=2m$ vertices, labeled as in Figure~\ref{fig:laddergraph}. Then the resistance distance between nodes $1$ and $2m$ is given by
	\[r(1,2m)= -1- \sqrt{3} + \frac{2 \sqrt{3}}{(1- (2-\sqrt{3})^{2m}},\]
  where $\alpha = 2-\sqrt{3}.$
  \end{theorem}

The Laplacian matrix has an easily describable structure.  The first super and sub diagonal entries alternate between 0 and $-1$ and the second super and sub diagonals are identically $-1$ as shown in~\eqref{equ:ladderlaplacian}
\begin{equation}\label{equ:ladderlaplacian}L_G=\left[\begin{array}{rrrrrrrr}
2 & -1 & -1 & 0  &   0& \dots & \dots &0\\
-1 & 2 & 0 & -1 &    0&  0& \ddots& \vdots\\
-1 &0 & 3 & - 1 & -1 &    0&  \ddots&  \vdots\\
0 &-1 &-1 & 3 & 0 & -1 &  \ddots&  0\\
0 & \ddots &\ddots & \ddots &  \ddots &  \ddots&  \ddots & 0\\
\vdots & \ddots & 0 & -1 &  -1 &  3&  0 & -1\\
\vdots & \ddots  &0 & 0 & -1 &  0 &  2&  -1 \\[1.7mm]
0  &\dots& \dots & 0 &0& -1 &  -1 & 2
\end{array}\right].\end{equation}

 \subsection{Fan Graphs}\label{sub:fans}
 
Another possible generalization of the path graph is obtained by joining the path graph with a singleton vertex.  This results in the so-called fan graph, which is shown in Figure~\ref{fig:fan}.

The Laplacian of the fan graph is given by
\[L_G=\left[\begin{array}{rrrrrrrr}
2 & -1 & 0 & 0  &  \ldots& \dots & 0 &-1\\
-1 & 3 & - 1 & 0 &    0&  0& \ddots& \vdots\\
0 &-1 & 3 & - 1 & 0 &    0&  \ddots&  \vdots\\
0 &0 &-1 & 3 & - 1 & 0 &  \ddots&  -1\\
\vdots & \ddots &\ddots & \ddots &  \ddots &  \ddots&  \ddots & -1\\
\vdots & \ddots & 0 & 0 &  -1 &  3&  -1 & -1\\
0 & \ddots  &0 & 0 & 0 &  -1 &  2&  -1 \\[1.7mm]
-1  &\dots& \dots & -1 &-1& -1 &  -1 & k-1
\end{array}\right].\]
 
 
Notice, that in contrast to the prior examples, the Laplacian is not banded, an important property when applying the algorithm of this paper.   Like many of the examples listed prior, a closed formula for the resistance distance between node $1$ and node $n$ is known.
\begin{prop}\cite{BapatWheels}\label{pro:bapatfans}
Let $k \geq 1$ be a positive integer. Then for $i= 1, \ldots k-1$, the resistance distance between node $i$ and node $k$ in the fan graph is given by 
\begin{gather*}
r(i,k) = \frac{F_{2(k-1-i)+1}+F_{2i-1}}{F_{2k-2}},
\end{gather*}
where $F_i$ is the $i$th Fibonacci number.
\end{prop}
\begin{figure}
\[\begin{array}{c}
\begin{tikzpicture}[line cap=round,line join=round,>=triangle 45,x=1.0cm,y=1.0cm,scale = 1.5]
\draw [line width=.8pt] (5.,3.)-- (4.191784354657935,3.848009081626107);
\draw [line width=.8pt] (4.39006772869402,1.9998427342084801)-- (3.8288722643919675,3.028138208643161);
\draw [line width=.8pt] (4.39006772869402,1.9998427342084801)-- (5.,3.);
\draw [line width=.8pt] (5.,3.)-- (3.8288722643919675,3.028138208643161);
\draw [line width=.8pt] (3.8288722643919675,3.028138208643161)-- (4.191784354657935,3.848009081626107);
\draw [line width=.8pt] (4.191784354657935,3.848009081626107)-- (5.,3.);
\draw [line width=.8pt] (5.,3.)-- (5.028138208643161,4.1711277356080325);
\draw [line width=.8pt] (5.028138208643161,4.1711277356080325)-- (4.191784354657935,3.848009081626107);
\draw [line width=.8pt] (5.,3.)-- (5.848009081626107,3.8082156453420652);
\draw [line width=.8pt] (5.848009081626107,3.8082156453420652)-- (6.171127735608033,2.9718617913568393);
\draw [line width=.8pt] (6.171127735608033,2.9718617913568393)-- (5.,3.);
\draw [line width=.8pt] (5.,3.)-- (5.5611954643020525,1.9717045255653174);
\draw [line width=.8pt] (5.5611954643020525,1.9717045255653174)-- (6.171127735608033,2.9718617913568393);
\begin{scriptsize}
\draw [fill=black] (5.,3.) circle (1.5pt);
\draw[color=black] (5,2.7) node {$1$};
\draw [fill=black] (4.191784354657935,3.848009081626107) circle (1.5pt);
\draw[color=black] (4.05,4.05) node {$4$};
\draw [fill=black] (3.8288722643919675,3.028138208643161) circle (1.5pt);
\draw[color=black] (3.5,3.08) node {$3$};
\draw [fill=black] (5.028138208643161,4.1711277356080325) circle (1.5pt);
\draw[color=black] (5,4.381775026148471) node {$5$};
\draw [fill=black] (5.848009081626107,3.8082156453420652) circle (1.5pt);
\draw[color=black] (6.1,4) node {$k-2$};
\draw [fill=black] (6.171127735608033,2.9718617913568393) circle (1.5pt);
\draw[color=black] (6.62,3.05) node {$k-1$};
\draw [fill=black] (4.39006772869402,1.9998427342084801) circle (1.5pt);
\draw [fill=black] (5.5611954643020525,1.9717045255653174) circle (1.5pt);
\draw[color=black] (5.25,1.8) node {$k$};
\draw[color=black] (4.2,1.8) node {$2$};
\draw [fill=black] (5.438073645134634,3.989671690475049) circle (.9pt);
\draw [fill=black] (5.233105926888897,4.08039971304154) circle (.9pt);
\draw [fill=black] (5.64304136338037,3.898943667908557) circle (.9pt);
\end{scriptsize}
\end{tikzpicture}
\end{array}\]
\caption{The fan graph on $k$ vertices}\label{fig:fan}
\end{figure}
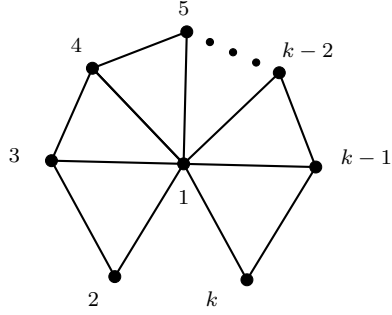
\subsection{Wheel Graphs}\label{sub:wheels}
An easy modification to the fan graph is to add an edge between node $1$ and node $n-1$.  This results in the so-called wheel graph as presented in Figure \ref{fig:wheels}.

The Laplacian of the wheel graph is given by
\[L_G=\left[\begin{array}{rrrrrrrr}
3 & -1 & 0 & 0  &  \ldots& 0 & -1 &-1\\
-1 & 3 & - 1 & 0 &    0&  0& \ddots& \vdots\\
0 &-1 & 3 & - 1 & 0 &    0&  \ddots&  \vdots\\
\vdots &0 &-1 & 3 & - 1 & 0 &  \ddots&  -1\\
\vdots & \ddots &\ddots & \ddots &  \ddots &  \ddots&  \ddots & -1\\
0 & \ddots & 0 & 0 &  -1 &  3&  -1 & -1\\
-1 & \ddots  &0 & 0 & 0 &  -1 &  3&  -1 \\[1.7mm]
-1  &\dots& \dots & -1 &-1& -1 &  -1 & k-1
\end{array}\right].\]

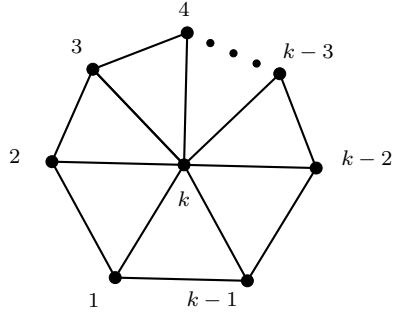
\begin{figure}
\[\begin{array}{c}
\begin{tikzpicture}[line cap=round,line join=round,>=triangle 45,x=1.0cm,y=1.0cm,scale = 1.5]
\draw [line width=.8pt] (5.,3.)-- (4.191784354657935,3.848009081626107);
\draw [line width=.8pt] (4.39006772869402,1.9998427342084801)-- (3.8288722643919675,3.028138208643161);
\draw [line width=.8pt] (4.39006772869402,1.9998427342084801)-- (5.5611954643020525,1.9717045255653174);

\draw [line width=.8pt] (4.39006772869402,1.9998427342084801)-- (5.,3.);
\draw [line width=.8pt] (5.,3.)-- (3.8288722643919675,3.028138208643161);
\draw [line width=.8pt] (3.8288722643919675,3.028138208643161)-- (4.191784354657935,3.848009081626107);
\draw [line width=.8pt] (4.191784354657935,3.848009081626107)-- (5.,3.);
\draw [line width=.8pt] (5.,3.)-- (5.028138208643161,4.1711277356080325);
\draw [line width=.8pt] (5.028138208643161,4.1711277356080325)-- (4.191784354657935,3.848009081626107);
\draw [line width=.8pt] (5.,3.)-- (5.848009081626107,3.8082156453420652);
\draw [line width=.8pt] (5.848009081626107,3.8082156453420652)-- (6.171127735608033,2.9718617913568393);
\draw [line width=.8pt] (6.171127735608033,2.9718617913568393)-- (5.,3.);
\draw [line width=.8pt] (5.,3.)-- (5.5611954643020525,1.9717045255653174);
\draw [line width=.8pt] (5.5611954643020525,1.9717045255653174)-- (6.171127735608033,2.9718617913568393);
\begin{scriptsize}
\draw [fill=black] (5.,3.) circle (1.5pt);
\draw[color=black] (5,2.7) node {$k$};
\draw [fill=black] (4.191784354657935,3.848009081626107) circle (1.5pt);
\draw[color=black] (4.05,4.05) node {$3$};
\draw [fill=black] (3.8288722643919675,3.028138208643161) circle (1.5pt);
\draw[color=black] (3.5,3.08) node {$2$};
\draw [fill=black] (5.028138208643161,4.1711277356080325) circle (1.5pt);
\draw[color=black] (5,4.381775026148471) node {$4$};
\draw [fill=black] (5.848009081626107,3.8082156453420652) circle (1.5pt);
\draw[color=black] (6.1,4) node {$k-3$};
\draw [fill=black] (6.171127735608033,2.9718617913568393) circle (1.5pt);
\draw[color=black] (6.62,3.05) node {$k-2$};
\draw [fill=black] (4.39006772869402,1.9998427342084801) circle (1.5pt);
\draw [fill=black] (5.5611954643020525,1.9717045255653174) circle (1.5pt);
\draw[color=black] (5.25,1.8) node {$k-1$};
\draw[color=black] (4.2,1.8) node {$1$};
\draw [fill=black] (5.438073645134634,3.989671690475049) circle (.9pt);
\draw [fill=black] (5.233105926888897,4.08039971304154) circle (.9pt);
\draw [fill=black] (5.64304136338037,3.898943667908557) circle (.9pt);
\end{scriptsize}
\end{tikzpicture}
 \end{array}\]
\caption{The wheel graph on $k$ vertices}\label{fig:wheels}
\end{figure}
\begin{prop}\cite{BapatWheels}\label{pro:bapatwheels}
The resistance distance between vertex $k$ and vertex $i$, $i\in\{1,\ldots, k-1\}$ in the wheel graph
is 
\[r(i,k)= \frac{F_{2k-2}^2}{F_{4k-4}-2F_{2k-2}}.\]
\end{prop}

\section{The General Algorithm}\label{sec:algorithm}
The algorithm to arrive at the asymptotic approximation of the resistance distance takes place in a sequence of steps presented in this section and illustrated  by considering the family of ladder graphs, whose Laplacian, illustrative figure, and resistance formula are presented in Section \ref{sub:ladder}. Several of the steps have a variety of subtleties and open problems which are discussed.    The following section then applies the algorithm to all other examples presented in Section \ref{sec:examples}.

\subsection*{Step 0 - Laplacian} 
 We begin the algorithm by considering the $n \times n$ Laplacian of the underlying graph.

\subsection*{Step 1 - The Laplace expansion}
 
Using the notation defined in Section \ref{sub:notation} and the Laplacian given by~\eqref{equ:ladderlaplacian}, we can begin the Laplace expansion by defining  
\begin{equation}\label{equ:abcdefinition}
A^{n  \; \times \; n} = L(\{1,n+2\}|\{1,n+2\}),
\qquad
B^{n-1 \; \times \; n-1} = A^{n \; \times \; n}(1|1), \qquad
C^{n-1 \; \times \; n-1} = A^{n \; \times \;  n }(1|3),
\end{equation}
and then noting that
\begin{equation}\label{equ:firstrowladder}
	Det(A^{n\; \times\; n})=  2 Det(B^{n-1 \; \times \; n-1}) -  Det(C^{n-1 \; \times \; n-1 }).
      \end{equation}

For a variety of reasons it is useful to minimize the superscripts. Towards that end we define the backward shift operator
$Y,$ whose action on a sequence   $\{G_n\}$ is given by
$$Y(G_n) = G_{n-1}.$$
Thus, equation \eqref{equ:firstrowladder} becomes
\begin{equation}\label{equ:firstrowladderY} 
	Det(A)=	  
    2Y Det(B) - Y Det (C), 
\end{equation}
the equation being valid for all $n$ where defined.

Rather than using upper case letters, in the sequel, we use a sequence notation for the matrices introduced. Letting $M(1) = A, M(2) = B, M(3) = C,$  \eqref{equ:firstrowladderY}
becomes
\begin{equation}\label{equ:firstrowladderYM}
	Det(M(1))=	  
    2Y Det(M(2)) - Y Det (M(3)).
\end{equation}

The program LaplaceExpand found in the Appendix, written in Mathematica 13.3, performs Laplace expansions on a collection of matrix families.   To keep track of the equations associated with the expansion, we find it convenient to collect the operations in a matrix $P$ with seven columns and as many rows as needed for the matrix families introduced. The first three rows of $P$ are given by 

\begin{equation}\label{equ:first3rowsofP}
\text{First 3 rows of $P$} =\begin{psmallmatrix}
 1  &   1   &   R   &   0   &   0   &   0   &   0 \\
 2  &   0   &   R   &   1   &   1   &   1   &   2Y \\
 3  &   0   &   C   &   1   &   1   &   3   &   -Y. 
 \end{psmallmatrix}.
 \end{equation} 

These three rows correspond to Equation \eqref{equ:firstrowladderYM} as follows: The equation has three matrices $M(1), M(2),$ and $ M(3)$; these indices 1,2 and 3 are found in column one of $P$. Recall, $M(2), M(3)$ arise from deleting rows and columns in $M(1).$ Hence we may say that $M(1)$ is the parent of $M(2)$ and $M(3),$ and this relationship is recovered in column four of $P$ which has a $1$ in it indicating the parent index of the matrix family of that row.  Columns five and six correspond to the rows and columns deleted from the parent matrix family. Since $M(3)^{n \; \times \; n}  = M(1)^{n+1 \; \times \; n+1}(1|3)$ the values 1 and 3 appear in columns five and six. Similarly the value 1 appears in columns five and six of row two corresponding to the definition of $M(2).$ Column seven stores the coefficients, $2Y$ and $-Y$ respectively of $M(2)$ and $M(3)$ in  \eqref{equ:firstrowladderYM}.

Columns two and three of $P$ store information useful for the LaplaceExpand algorithm. The R and C in column three indicate whether the expansion that is eventually performed on these matrices uses the first row or column. The algorithm uses the convention that if the number of non-zero elements in the first column of the matrix it is expanding is strictly less than the number of elements in the first row of that matrix, then the expansion is done on the first column; otherwise the first row is used. 

As the Laplace expansion progresses column two keeps track of what has been expanded and what has not been expanded. For example, after performing the expansion indicated by  \eqref{equ:firstrowladderYM} and producing \eqref{equ:first3rowsofP}, $M(1)$ is expanded and hence there is a 0 in column two, while $M(2)$ and $M(3)$ have not been expanded yet and hence there is a 1 in those rows indicating that in future steps $M(2)$ and $M(3)$ must be expanded. The program terminates when column two is identically zero.

 The entire collection of Laplace expansions for the ladder graph is given by  
  
$$P=\begin{psmallmatrix}  
 1 & 0 & \text{R} & 0 & 0 & 0 & 0 \\
 2 & 0 & \text{R} & 1 & 1 & 1 & 2 Y \\
 3 & 0 & \text{C} & 1 & 1 & 3 & -Y \\
 4 & 0 & \text{R} & 2 & 1 & 1 & 3 Y \\
 5 & 0 & \text{R} & 2 & 1 & 2 & Y \\
 6 & 0 & \text{C} & 2 & 1 & 3 & -Y \\
 7 & 0 & \text{R} & 3 & 2 & 1 & Y \\
 2 & 0 & 0 & 4 & 1 & 1 & 3 Y \\
 8 & 0 & \text{C} & 4 & 1 & 3 & -Y \\
 2 & 0 & 0 & 5 & 1 & 1 & -Y \\
 9 & 0 & \text{C} & 5 & 1 & 3 & -Y \\
 10 & 0 & \text{C} & 6 & 1 & 1 & -Y \\
 11 & 0 & \text{R} & 6 & 2 & 1 & Y \\
 2 & 0 & 0 & 7 & 1 & 1 & 3 Y \\
 12 & 0 & \text{C} & 7 & 1 & 2 & Y \\
 7 & 0 & 0 & 8 & 2 & 1 & Y \\
 5 & 0 & 0 & 9 & 1 & 1 & -Y \\
 5 & 0 & 0 & 10 & 2 & 1 & Y \\
 4 & 0 & 0 & 11 & 1 & 1 & 3 Y \\
 13 & 0 & \text{C} & 11 & 1 & 2 & Y \\
 4 & 0 & 0 & 12 & 1 & 1 & -Y \\
 2 & 0 & 0 & 13 & 1 & 1 & -Y \\ 
\end{psmallmatrix}.$$

Important for the termination of the program is that if a matrix is repeated, it is recognized.  We illustrate this by  row eight of $P$, (the row beginning after the row beginning with 7). This row begins with a 2 not an 8 and corresponds to the matrix family equation  
$$
M(2) = M(4)(1|3).
$$
The $2$ and $4$  in this equation  correspond, as indicated above, to the entries in columns one and four respectively. What has happened here is that upon expanding $M(4)$ we do not arrive at a new matrix family, but rather, we arrive at a matrix family formerly encountered. The matrices $M(4)^{6 \; \times \; 6}$ and $M(2)^{5 \; \times \; 5} $ are as follows.
\[
 M(4)^{6 \; \times \; 6}=\begin{psmallmatrix}
 3& 0& -1& 0& 0& 0\\
 0& 3& -1& -1& 0& 0\\
 -1& -1& 3& 0& -1& 0\\
 0& -1& 0& 3& -1& -1\\
 0& 0& -1& -1& 3& 0\\
 0& 0& 0& -1& 0& 2
\end{psmallmatrix}
\qquad
M(2)^{5 \; \times \; 5} = \begin{psmallmatrix}
 3& -1& -1& 0& 0\\
 -1& 3& 0& -1& 0\\
 -1& 0& 3& -1& -1\\
 0& -1& -1& 3& 0\\
 0& 0& -1& 0& 2
\end{psmallmatrix}.
\]
    It is easily inspected that $M(4)^{n+1 \; \times \; n+1}(1|1)= M(2)^{n \; \times \; n}$ for several consecutive values of $n.$ 
However, that does not prove equality; therefore, the program prints all matrices and allows inspection of all identities. For the inspection to succeed in confirming identities in determinant families,   the matrices must be of a minimum size that preserve all features of the matrix family. This minimum size is inputted to the program through a variable MinimumSize, which for the ladder Graph family is set equal to 5, allowing the patterns of threes on the diagonal as well as the alternating $0$-s and $-1$-s to be identified. For each family, the minimum size must be set consistent with requirements of the structure of the Laplacian.

Currently we have no way to prove that this process converges, that is, eventually does not produce any new matrix families.  However in all cases examined, convergence does take place. We might even heuristically argue  that since the definition of the Laplacians are finite in nature (a finite number of patterns are used to define the Laplacian) it seems reasonable that the Laplace expansion process should converge.

To be more precise, on the cases examined:
\begin{itemize}
    \item If the matrix family is uniformly banded,  then convergence takes place in all examples examined.
    \item If the matrix family is uniformly banded with a finite number of exceptions (such as happens with the fan and wheel graph families discussed below), convergence must be assisted by some manual manipulations (which seems to be amenable to further programming).
    \item If the matrix family is not uniformly banded, we aren't certain of whether convergence takes place. One such family is the family of $n$-grids whose underlying $n$-th graph consists of $n$ rows of upright oriented triangles with the base row containing $n$ triangles (in a Cartesian representation). While, the Laplacian of any individual member of this family is banded, the number of non-zero diagonals is increasing without bound as $n$ goes to infinity. We have not attempted to explore this further.  
\end{itemize}

We leave the issue of convergence as an open problem, with the criteria for assuring convergence also being open.

\begin{conjecture}~\label{conj:6} For matrix families satisfying specified criteria,the Laplace expansion process outlined always converges. \end{conjecture}

 \subsection*{Step 2 - The Simultaneous Identity System} We wish to solve the system of simultaneous identities in determinant families implied by $P.$ To do this we need to convert $P$ to a matrix $Q$ which lays out these identities.

We have
$$Q = 
\begin{psmallmatrix} 
 &0& 2 Y& -Y& 0& 0& 0& 0& 0& 0& 0& 
  0& 0& 0\\
 &0& 0& 0& 3 Y& Y& -Y& 0& 0& 0& 0& 
  0& 0& 0\\
 &0& 0& 0& 0& 0& 0& Y& 0& 0& 0& 0& 
  0& 0\\
 &0& 3 Y& 0& 0& 0& 0& 0& -Y& 0& 0& 
  0& 0& 0\\
 &0& -Y& 0& 0& 0& 0& 0& 0& -Y& 0& 0&
   0& 0\\
 &0& 0& 0& 0& 0& 0& 0& 0& 0& -Y& Y& 
  0& 0\\
 &0& 3 Y& 0& 0& 0& 0& 0& 0& 0& 0& 0&
   Y& 0\\
 &0& 0& 0& 0& 0& 0& Y& 0& 0& 0& 0& 
  0& 0\\
 &0& 0& 0& 0& -Y& 0& 0& 0& 0& 0& 0& 
  0& 0\\
 &0& 0& 0& 0& Y& 0& 0& 0& 0& 0& 0& 
  0& 0\\
 &0& 0& 0& 3 Y& 0& 0& 0& 0& 0& 0& 0&
   0& Y\\
 &0& 0& 0& -Y& 0& 0& 0& 0& 0& 0& 0& 
  0& 0\\
 &0& -Y& 0& 0& 0& 0& 0& 0& 0& 0& 0& 
  0& 0 
\end{psmallmatrix}.$$

In Matrix Q, row $i$ corresponds to the identity of matrix families
\begin{equation}\label{equ:rowofQ}
    Det(M(i)) = \sum_{all \; j} Q_{i,j} Det(M(j)).
\end{equation}
For example,  the first row of $Q$ corresponds to the identity
$$
    Det(M(1)) = 2Y Det(M(2)) - Y Det(M(3)),
$$
this identity in matrix families being valid for any size $n.$  This last equation is identical to \eqref{equ:firstrowladderYM}. A final step in the program LaplaceExpand produces the matrix $Q$ from the matrix $P.$

\subsection*{Step 3 - System Reduce} 
Prior to talking about solving the system of determinant identities we introduce some standard terminology illustrated using the Fibonacci sequence whose underlying \textit{minimal recursion} is $F_{n} = F_{n-1} + F_{n-2},$ with \textit{characteristic polynomial} $X^2-X-1$ and with \textit{annihilator} $Y^2 -Y - 1$ (with the operator $Y^2-Y-1$ being called an annihilator because when it is applied to the Fibonacci sequence it yields the identically 0 sequence~\cite{Boole}). Throughout the paper we use the terms recursion, characteristic polynomial, and annihilator interchangeably. Further, recall that because $\Z[Y]$ is a principal ideal domain, the minimal polynomial generates the ideal of all characteristic polynomials. So, for example,
the polynomial  $(X-1)(X^2-X-1) = X^3 -2 X^2 + 1$ corresponding to the recursion 
$F_{n} = 2 F_{n-1} - F_{n-3}$ is satisfied  by all Fibonacci numbers, and could, with appropriate initial values, be alternatively used to define the Fibonacci numbers.

Returning to the matrix $Q$ we must \textit{solve} the system by which we mean finding a single annihilator of one of the determinant families.  To accomplish this we need a method  of reducing $Q$.  

The program SystemReduce automates this process. The program loops once through each row in $Q$.  Recall that \eqref{equ:rowofQ} presents the identity in determinant families indicated by row $i$ of $Q$.   This equation allows us to substitute its right side for every occurrence of $M(i)$ in the other rows of $Q.$ Doing so eliminates the column corresponding to $M(i)$ or more precisely makes it 0. 

The result of this algorithm is an equivalent system of determinant identities which we call in the sequel $R$. For the family of ladder graphs we have

\begin{equation}\label{equ:R} R=\begin{psmallmatrix}  &0&	0&	0&	6 Y^2-8 Y^4+3 Y (-2 Y^3+3 Y^5)			&&2 Y^2-3 Y^4+Y (2 Y^3-3 Y^5)&	0&	0&	0&	0&	0&	0&	0&	&Y (-2 Y^3+3 Y^5)\\
&0&	0&	0&	3 Y-3 Y^3							&&Y+Y^3&						0&	0&	0&	0&	0&	0&	0&	&-Y^3\\
&0&	0&	0&	8 Y^3-9 Y^5	 						&&3Y^3+3 Y^5&					0&	0&	0&	0&	0&	0&	0&	&-3 Y^5\\
&0&	0&	0&	9 Y^2-8 Y^4+3Y (-3 Y^3+3 Y^5)			&&3 Y^2-3 Y^4+Y (3 Y^3-3 Y^5)&	0&	0&	0&	0&	0&	0&	0&	&Y (-3 Y^3+3 Y^5)\\
&0&	0&	0&	-3 Y^2+3 Y^4							&&-Y^4&						0&	0&	0&	0&	0&	0&	0&	&Y^4\\
&0&	0&	0&	3 Y^2								&&-Y^2&						0&	0&	0&	0&	0&	0&	0&	&Y^2\\
&0&	0&	0&	8 Y^2-9 Y^4							&&	3 Y^2+3 Y^4&				0&	0&	0&	0&	0&	0&	0&	&-3 Y^4\\
&0&	0&	0&	8 Y^3-9 Y^5							&&	3 Y^3+3 Y^5&				0&	0&	0&	0&	0&	0&	0&	&-3 Y^5\\
&0&	0&	0&	0									&&-Y&							0&	0&	0&	0&	0&	0&	0&	&0\\
&0&	0&	0&	0									&&Y&							0&	0&	0&	0&	0&	0&	0&	&0\\
&0&	0&	0&	3 Y									&&0&							0&	0&	0&	0&	0&	0&	0&	&Y\\
&0&	0&	0&	-Y									&&0&							0&	0&	0&	0&	0&	0&	0&	&0\\
&0&	0&	0&	-3 Y^2+3 Y^4							&&-Y^2-Y^4&					0&	0&	0&	0&	0&	0&	0&	&Y^4 \end{psmallmatrix}.
\end{equation}

While this system is not in the form of one annihilator for any determinant family some mathematical theorems, presented in the next step, allow us to solve it. 

\subsection*{Step  4 - Obtaining Annihilators}  Throughout this step, upper case letters from the beginning of the  English alphabet, $A,B,C,\dotsc$ with and without subscripts will denote annihilators which are polynomials in $Y;$ lower case letters from the end of the English alphabet, $x,y,z$ with subscripts will denote individual members of determinant families; and lower case  letters from the end of the alphabet without subscripts will denote the entire determinant family, $x = \{x_n\}_{n \ge c}$ with $c$ some constant (though not necessarily 1 or 0). 

The following two propositions are sufficient to calculate annihilators of all examples studied in this paper.

\begin{proposition}\label{pro:2x2} (i) For the system 
\begin{equation}\label{equ:2x2}
A' x = B' x + C' y \qquad D' y= E' x  + F' y,\end{equation}  
the operator
\begin{equation}\label{equ:2x2solved}
(D' - F')(A' - B') - C' E'
\end{equation}
annihilates both $x,$ $y,$ and any linear combination of them with polynomial coefficients in $Y$.\\
(ii) The system
\begin{equation}\label{equ:1x1}
    A'x = B'x
\end{equation}
is annihilated by 
\begin{equation}\label{equ:1x1solved}
A' - B' .
\end{equation}

\end{proposition}
\begin{proof}  The proof of (ii) is trivial and omitted. It is stated for purposes of completeness. We proceed to prove (i). The given system is equivalent to the system $(A' - B') x = C'y,$ and $(D'-F')y = E' x.$  Multiplying both sides of the first equation of this equivalent system  by $(D'-F')$ we obtain
$(D'-F')(A'-B') x = C' (D'-F') y = C' E' x $ implying that $(D'-F')(A'-B') - C' E'$ annihilates $x.$ A similar proof shows that $(D'-F')(A'-B') - C' E'$ also annihilates $y.$ It follows that   $(D'-F')(A'-B') - C' E'$ annihilates any linear combination of $x,y$ with polynomial coefficients. \end{proof}

\begin{proposition}\label{pro:3x3} The system 
\begin{equation}\label{equ:3x3}
x= Ax + By +C z,\qquad y = D x + E y + F z, \qquad z = G x + H y + I z\end{equation}
can be reduced to the equivalent system
\eqref{equ:2x2}  where
\begin{multline}\label{3x3solved}
 A'=1-E, B'=(1-E)A+BD, C'=(1-E)C+BF,\\
D'=(1-E),E'=(1-E)G+HD, F'=(1-E)I+HF,
\end{multline}
and hence \eqref{equ:2x2solved} annihilates $x,y,z.$ \end{proposition}
\begin{proof} The second identity is equivalent to $(1-E)y = D x + F z.$ If  we now multiply both sides of the other two identities by $(1-E)$ we obtain the following equivalent system of
two identities in two determinant families: $(1-E) x = (1-E) A x + B (1-E)y + C (1-E) z = ((1-E)A + BD) x + ((1-E) C + BF) z$ and
$(1-E) z = ((1-E)G + HD) x +((1-E) I + HF) z.$ The result now follows by 
\eqref{equ:2x2}-\eqref{equ:2x2solved}.
\end{proof}

The following theorem, not needed in this paper, generalizes the previous two propositions for $n$ families of determinant identities in $n$ determinant families.

\begin{theorem}\label{them:allfam} The system $x_i = \sum_{j=1}^n  A_{i,j} x_j$ of $n$ determinant family identities in $n$ determinant families can be solved, that is, there is a polynomial operator that simultaneously annihilates all determinant families and hence also annihilates  any linear combinations of them with polynomial coefficients. \end{theorem}
\begin{proof} We will refer to the identity $x_m = \sum_{j=1}^n A_{m,j}$ as the $m$-th equation or $m$-th identity. 
The proof is by induction, the base case being the $3 \times 3 $ case.
Assume, using an induction assumption that the theorem true for the case $n-1 \ge 3;$ we proceed to prove it for the case $n$.  The $n$-th equation is the identity
$x_n = \sum_{j=1}^n A_{n,j} x_j,$ which is equivalent to the identity $(1-A_{n,n}) x_n = \sum_{j=1}^{n-1} A_{n,j} x_j.$ If we multiply both sides of the identities $1,\dotsc, n-1$
by $(1-A_{n,n})$  and then replace  all instances of $(1-A_{n,n}) x_n $ with $\sum_{j=1}^{n-1} A_{n,j} x_j$  we obtain an equivalent system of $n-1$ identities in $n-1$ determinant families. 
But by the induction assumption this system can be solved. \end{proof}

\begin{example}
We use the two propositions to reduce \eqref{equ:R} to a single annihilator. Using the notation of Proposition \ref{pro:3x3},  let 
$
A=R^4_{4,4},
B=R^4_{4,5},
C=R^4_{4,13},
D=R^4_{5,4},
E=R^4_{5,5},
F=R^4_{5,13},
G=R^4_{13,4},
H=R^4_{13,5},
I=R^4_{13,13}.$
The \eqref{equ:3x3} is satisfied and the conclusion of Proposition \ref{pro:3x3}   states that  \eqref{equ:2x2solved} annihilates $M(4), M(5), M(13)$. This annihilator in factored form is
$$
-\left((Y-1) (Y+1) \left(Y^4+1\right) \left(Y^4-4 Y^2+1\right)^2\right)
$$

However, $R$ shows that all determinant families are linear combinations (in polynomials of $Y$) of $M(4), M(5), M(13)$. Hence \eqref{equ:2x2solved} annihilates all determinant families. In particular it annihilates $M(1)$ which by \eqref{equ:Bapat} is the numerator in a formula for the effective resistance.

Going a step further, all Laplace expansions were done either on the first row or column. Hence, provided the underlying matrices of the numerator and denominator of \eqref{equ:Bapat} are banded, any sequence of Laplace expansions that annihilates the sequence 
$L^n(\{1,n\}|\{1,n\})$ also annihilates  the sequence $L^n(1|1).$ In other words, if the underlying Laplacian is banded, we have arrived at the annihilator (characteristic polynomial) of both the numerator and denominator in \eqref{equ:Bapat}. If the Laplacian is not banded, we may have to apply SystemReduce twice, once for the numerator and once for the denominator.
\end{example}

Again, we have no way of proving that this process always works. But in all examples thus far examined it has worked. We state this as an open conjecture.
We again leave the specification of criteria as part of the conjecture as we did in Conjecture~\ref{conj:6}. 
\begin{conjecture}
For any family of graphs with specified criteria, LaplaceExpand, SystemReduce, and the mathematical theorems suffice to find a single annihilator for both the numerator and denominator of formula~\eqref{equ:Bapat}.
\end{conjecture}

\subsection*{Step 5 - Minimal Annihilator}
 This annihilator found  in Step 4 is not the minimal annihilator. By the remarks above, the annihilator found belongs to the ideal generated by the minimal annihilator. That means, we only need check all polynomial divisors of the annihilator to see which one is the minimal annihilator. The following proposition gives details.

 \begin{proposition}\label{pro:minimalpolynomial} If (i) $A(Y)$ annihilates $x=\{x_n,\}_{n\ge c},$  (ii)  $A(Y)= B(Y) C(Y),$ and (iii) $C(x_n) = 0, \text{ for } n=n_1,n_1+1,\dotsc, n_1+deg(B)-1$, then $C$ annihilates X. \end{proposition}
\begin{proof}   We know that $A=BC$ annihilates $x.$ Then $B$ annihilates the sequence $C(x).$ Since $B$ is of degree $deg(B),$ corresponding to a recursion of order $deg(B)$ satisfied by the sequence $\{C(x)\},$  values of this sequence are determined by any $deg(B)$ consecutive values. It follows that if $C(x)=0$ on $deg(B)$ consecutive terms then it must identically equal 0. \end{proof}

 \begin{example}
Continuing with the ladder example, according to the proposition, we simply test all proper factors of the annihilator obtained in Step 4, eliminating those that do not annihilate the sequence of determinants in the numerator and denominator.
We have that 
The minimal annihilator for both the numerator and denominator is 
\[\left(Y-1)(Y+1)( Y^4-4 Y^2+1 \right)^2.\]

\end{example}

There are several subtleties associated with this annihilator. By \eqref{equ:abcdefinition}, this annihilator annihilates the sequence $Det(A^{n \; \times \; n}).$ However, it only does so for $n \ge 10$ for the numerator sequence and for $n \ge 13$ for the denominator sequence.  This is because it is not clear what a ladder graph of size 1 or 2 should be and whether it fits in with the rest of the family.

Another subtlety is that the ladder graph is defined for even $n$ while $A^{n \; \times \; n}$ is defined for both even and odd $n.$ This can be justified  by the underlying graphs, since the ladder graph for an odd number vertices does not look like a ladder. However, if we have the characteristic polynomial for a sequence, we can easily derive the characteristic polynomial for subsequences defined by arithmetic progression of indices \cite{Jarden}. In this case the minimal polynomial (for both the  numerator and denominator) is 
$(Y-1) (Y^4 - 4Y^2 +1)^2.$

As a final subtlety we note that the roots of any of these annihilators may be obtained in closed form, which is useful when calculating closed  formula as discussed in the next step.

\subsection*{Step 6 - Binet Forms} Having found the minimal polynomials for the sequence of determinants in the numerator and denominator of \eqref{equ:Bapat},  it is straightforward  to compute Binet forms for the value of the $n$-th determinant. Two standard approaches to obtaining these formula are generating functions or solving a system of linear equations for the unknown coefficients. These closed Binet forms allow computation of exact and asymptotic formula for resistance.

 \color{black} 

\section{Applications of the Algorithm}

This section applies the algorithm in Section~\ref{sec:algorithm} to the examples presented in Section~\ref{sec:examples}. For each example the matrices $P, Q, R,$ outputs of the software programs in Appendix I, are provided or described. We then indicate how application of the mathematical theorems provide minimal polynomials. For certain examples we indicate minor modifications to address parity.

\subsection{Path Graph.} Applying LaplaceExpand to the appropriate matrix family, $L(\{1,n\}|\{1,n\}),$ the numerator of \eqref{equ:Bapat}, with $L$ the Laplacian of the Path Family, we obtain
\[
P=\begin{psmallmatrix}
1 & 0 & R & 0 & 0 & 0 &0\\
1 & 0 & 0 & 1 & 1 & 1 & 2Y\\
2 & 0 & C & 1 & 1 &2 & Y\\
1 & 0 & 0 & 2 & 1 & 1 & -Y \end{psmallmatrix}; \qquad
Q = \begin{psmallmatrix}
2Y & Y \\
-Y  & 0 \end{psmallmatrix} \qquad
R = \begin{psmallmatrix}2Y-Y^2 & 0\\ -Y & 0.
\end{psmallmatrix}
\]
As shown in Step 6 of Section \ref{sec:algorithm}, $P,Q,R$ for the denominator $L(1|1)$ is identical with the $P,Q,R$ of the numerator.
By  \eqref{equ:1x1}-\eqref{equ:1x1solved}
  we infer that an annihilator, in fact the minimal annihilator, for 
$L(\{1,n\}|\{1,n\})$, is $ Y^2-2Y+1,$ and by Proposition~\ref{pro:minimalpolynomial}, the minimal annihilator for 
$L(\{1,n\}|\{1,n\})$ is 1.  The numerator recursion is valid for $n \ge 3,$ while the denominator recursion is valid for $n \ge 2.$
It is then straightforward to derive the resistance distances described in Section \ref{sec:examples}.

\subsection{Linear 2--Tree}
Applying LaplaceExpand to the appropriate matrix family we obtain
\[
P=
\begin{psmallmatrix}
 1 & 0 & \text{R} & 0 & 0 & 0 & 0 \\
 2 & 0 & \text{R} & 1 & 1 & 1 & 3 Y \\
 3 & 0 & \text{C} & 1 & 1 & 2 & Y \\
 4 & 0 & \text{C} & 1 & 1 & 3 & -Y \\
 2 & 0 & 0 & 2 & 1 & 1 & 4 Y \\
 5 & 0 & \text{C} & 2 & 1 & 2 & Y \\
 6 & 0 & \text{C} & 2 & 1 & 3 & -Y \\
 2 & 0 & 0 & 3 & 1 & 1 & -Y \\
 3 & 0 & 0 & 3 & 2 & 1 & Y \\
 3 & 0 & 0 & 4 & 1 & 1 & -Y \\
 7 & 0 & \text{R} & 4 & 2 & 1 & Y \\
 2 & 0 & 0 & 5 & 1 & 1 & -Y \\
 3 & 0 & 0 & 5 & 2 & 1 & Y \\
 3 & 0 & 0 & 6 & 1 & 1 & -Y \\
 8 & 0 & \text{R} & 6 & 2 & 1 & Y \\
 2 & 0 & 0 & 7 & 1 & 1 & 4 Y \\
 9 & 0 & \text{C} & 7 & 1 & 2 & Y \\
 2 & 0 & 0 & 8 & 1 & 1 & 4 Y \\
 10 & 0 & \text{C} & 8 & 1 & 2 & Y \\
 2 & 0 & 0 & 9 & 1 & 1 & -Y \\
 2 & 0 & 0 & 10 & 1 & 1 & -Y \\
\end{psmallmatrix},
\qquad
Q=
\begin{psmallmatrix}
 0 & 3 Y & Y & -Y & 0 & 0 & 0 & 0 & 0 & 0 \\
 0 & 4 Y & 0 & 0 & Y & -Y & 0 & 0 & 0 & 0 \\
 0 & -Y & Y & 0 & 0 & 0 & 0 & 0 & 0 & 0 \\
 0 & 0 & -Y & 0 & 0 & 0 & Y & 0 & 0 & 0 \\
 0 & -Y & Y & 0 & 0 & 0 & 0 & 0 & 0 & 0 \\
 0 & 0 & -Y & 0 & 0 & 0 & 0 & Y & 0 & 0 \\
 0 & 4 Y & 0 & 0 & 0 & 0 & 0 & 0 & Y & 0 \\
 0 & 4 Y & 0 & 0 & 0 & 0 & 0 & 0 & 0 & Y \\
 0 & -Y & 0 & 0 & 0 & 0 & 0 & 0 & 0 & 0 \\
 0 & -Y & 0 & 0 & 0 & 0 & 0 & 0 & 0 & 0 \\
\end{psmallmatrix} 
\]

\[R= 
\begin{psmallmatrix}
  0 & Y^4-4 Y^3+3 Y & Y^2+Y & 0 & 0 & 0 & 0 & 0 & 0 & 0 \\
 0 & Y^4-4 Y^3-Y^2+4 Y & 2 Y^2 & 0 & 0 & 0 & 0 & 0 & 0 & 0 \\
 0 & -Y & Y & 0 & 0 & 0 & 0 & 0 & 0 & 0 \\
 0 & 4 Y^2-Y^3 & -Y & 0 & 0 & 0 & 0 & 0 & 0 & 0 \\
 0 & -Y & Y & 0 & 0 & 0 & 0 & 0 & 0 & 0 \\
 0 & 4 Y^2-Y^3 & -Y & 0 & 0 & 0 & 0 & 0 & 0 & 0 \\
 0 & 4 Y-Y^2 & 0 & 0 & 0 & 0 & 0 & 0 & 0 & 0 \\
 0 & 4 Y-Y^2 & 0 & 0 & 0 & 0 & 0 & 0 & 0 & 0 \\
 0 & -Y & 0 & 0 & 0 & 0 & 0 & 0 & 0 & 0 \\
 0 & -Y & 0 & 0 & 0 & 0 & 0 & 0 & 0 & 0 \\
\end{psmallmatrix}.
\]

By Proposition \ref{pro:minimalpolynomial} the 
minimal polynomials for
the numerator and denominator
are $(X+1)(X^2-3X+1)^2$ and $X^2-3X+1$
respectively. The corresponding 
recursions are valid for determinants
with indices greater than or equal to
7 and 5 respectively. The Binet forms
can be constructed in explicit algebraic form
with the roots of these polynomials
which are $1,$ and $\frac{3 \pm \sqrt{5}}{2}$
with various multiplicities. Using
\eqref{equ:Bapat} the results of Theorem \ref{the:2tree}
may be derived.

Interestingly, Theorem \ref{the:2tree}  is formulated in terms
of the most familiar of the order two recursions, the
Fibonacci and Lucas numbers, while the formulation
of results using \eqref{equ:Bapat} directly uses
the recursion $G_n= 3G_{n-1} - G_{n-2}.$ This reflects
the fact that order 2 recursions may be formulated in
terms of Fibonacci and Lucas numbers.

\subsection{The Fan Graph Revisited} Applying LaplaceExpand to the numerator in \eqref{equ:Bapat}, $L(\{1|n\}|\{1,n\})$   we obtain
$$P=\begin{psmallmatrix} 
 1 & 0 & \text{R} & 0 & 0 & 0 & 0 \\
 1 & 0 & 0 & 1 & 1 & 1 & 3 Y \\
 2 & 0 & \text{C} & 1 & 1 & 2 & Y \\
 1 & 0 & 0 & 2 & 1 & 1 & -Y \\ 
\end{psmallmatrix},
\qquad
Q=\begin{psmallmatrix}
 3 Y & Y \\
 -Y & 0 \\
\end{psmallmatrix},
\qquad
R=\begin{psmallmatrix}
  3 Y-Y^2 & 0 \\
 -Y & 0 \\
 \end{psmallmatrix}.
$$

By \eqref{equ:1x1} - \eqref{equ:1x1solved}
  the characteristic polynomial is $Y^2-3Y+1,$ which is also the minimal polynomial, with roots $3 \pm \sqrt{5}.$  The corresponding recursion is valid for $n \ge 3.$  

Throughout the paper we have, for the denominator in \eqref{equ:Bapat}, used $L(1,1)$. However, $L(1|1)$ is not banded. This creates problems for the various programs as currently written since $L(1|1)(1,n)=-1,$ which would give a coefficient of $-(-1)^n$ in a Laplace Expansion across the first row. This creates problems for the $Y$ operator since when a Laplace Expansion is done $n$ decreases by 1. The software could be fixed to deal with coefficients involving $(-1)^n$ but we have not attempted to do this. Instead, for the fan and wheel graphs other fixes are presented. 

For the fan graph family, as indicated in the narrative accompanying \eqref{equ:Bapat}, we may use  $L(n|n)$ or $L(1|1)$ for the  denominator.   The advantage of using $L(n|n)$ is that it is banded. 

However, a second problems arises in that  the argument used in Step 6 of Section \ref{sec:algorithm}-- stating that the Laplace Expansion is identical for the numerator and denominator in \eqref{equ:Bapat} -- no longer applies as the first row and column of $L(n|n)$ is different from the first row and column of $L(\{1,n\}|\{1,n\}).$ 

However, the simple fix for this second problem is to apply SystemReduce a second time, this time to $L(n|n).$ The resulting matrices are as follows:

$$P=\begin{psmallmatrix} 
  1 & 0 & \text{R} & 0 & 0 & 0 & 0 \\
 2 & 0 & \text{R} & 1 & 1 & 1 & 2 Y \\
 3 & 0 & \text{C} & 1 & 1 & 2 & Y \\
 2 & 0 & 0 & 2 & 1 & 1 & 3 Y \\
 4 & 0 & \text{C} & 2 & 1 & 2 & Y \\
 2 & 0 & 0 & 3 & 1 & 1 & -Y \\
 2 & 0 & 0 & 4 & 1 & 1 & -Y \\
\end{psmallmatrix}, 
Q=\begin{psmallmatrix} 
 0 & 2 Y & Y & 0 \\
 0 & 3 Y & 0 & Y \\
 0 & -Y & 0 & 0 \\
 0 & -Y & 0 & 0 \\
\end{psmallmatrix}, 
R=\begin{psmallmatrix} 
0 & 2 Y-Y^2 & 0 & 0 \\
 0 & 3 Y-Y^2 & 0 & 0 \\
 0 & -Y & 0 & 0 \\
 0 & -Y & 0 & 0 \\ 
\end{psmallmatrix}.
$$

Applying \eqref{equ:1x1}-\eqref{equ:1x1solved} we obtain the characteristic (also the minimal) polynomial  $X^2-3X+1,$ with the corresponding recursion valid for $n \ge 4.$

Using the minimal polynomials of the numerator and denominator it is routine to obtain Proposition \ref{pro:bapatfans}. Again, we note that  the results using the methods in this paper are formulated in terms of recursions whose Binet form uses $3+\sqrt{5},$ instead of Fibonacci numbers, though clearly the two formulations are equivalent. 

\subsection{The Wheel Graph Revisited} Applying LaplaceExpand to the underlying matrix family of the numerator of \eqref{equ:Bapat}, we obtain  the following matrices.

\[P=\begin{psmallmatrix}
 1 & 0 & \text{R} & 0 & 0 & 0 & 0 \\
 1 & 0 & 0 & 1 & 1 & 1 & 3 Y \\
 2 & 0 & \text{C} & 1 & 1 & 2 & Y \\
 1 & 0 & 0 & 2 & 1 & 1 & -Y  
\end{psmallmatrix},
Q=\begin{psmallmatrix}
 3 Y & Y \\
 -Y & 0 \\
 \end{psmallmatrix},
R=\begin{psmallmatrix}
 3 Y-Y^2 & 0 \\
 -Y & 0 \\
 \end{psmallmatrix}.\]

The matrix $R$ for the wheel graph is identical to the matrix $R$ for the fan graph. The characteristic polynomial is therefore the same; the characteristic polynomial is also the minimal polynomial. The corresponding recursion is valid for $n \ge 3.$

However, neither of the two potential denominators for the numerator in \eqref{equ:Bapat} is banded, and as indicated above, the software does not readily apply to them.  

For the wheel graph family we fix this problem by manually expanding $L(n|n)$; the $(-1)^n$ vanishes after a few manual reductions. Letting $A=L(n|n),$  and manually performing successive Laplace Expansions, we obtain
  \begin{align*}
&A &=& 3Y A(1|1) + Y A(1|2) + (-1)^n Y A(1|3)\\		&A(1|2)  &= &-Y A(1|1) -(-1)^n (-1)^{n-2}	& \; \\		 
&A(1|3) &=& -(-1)^{n-2} - (-1)^n Y A(1|1)
\end{align*}
We can manually solve this system to obtain
$$A=A(1|1) \left(3 Y-2 Y^2\right)-2 Y.$$

Clearly $Y$ is annihilated by $X-1.$ Since $M(1|1)=
L(\{1,n\}|\{1,n\}),$ the numerator in \eqref{equ:Bapat}, $M(1|1)$ is annihilated by   $X^2-3X+1.$  Thus, the  characteristic (and also the minimal) polynomial would be $(X^2-3X+1) (X-1).$ The corresponding recursion is valid for $n \ge 6.$ Having obtained the minimal polynomials for the numerator and denominator in \eqref{equ:1x1solved} we can derive the resistance distance stated in Proposition~\ref{pro:bapatwheels}.

\subsection{The Linear 3--Tree Revisited}   As mentioned in the introduction, the key strength of the approach
of this paper to calculating resistances is its semiautomatic nature. It takes a few seconds to run the program which requires 201 Laplace expansions and introduces 80 matrix families. Thus the matrix Q is of size $80 \times 80.$ The matrix R has only three non-zero columns, at positions three, six, and forty-eight, allowing application of Proposition \ref{pro:3x3},
which only requires entries from the 3 rows (3,6,48) and the 3 columns (3,6,48), as presented in the following reduced matrix $R'.$
\[ 
R'=
\begin{psmallmatrix} 
 2 Y^8+6 Y^6-12 Y^5-15 Y^4+4 Y^3+Y^2+Y & -2 Y^8-Y^7-14 Y^6+54 Y^5+6 Y^4-5 Y^3-6 Y^2 & -9 Y^5-Y^4 \\
 -2 Y^7+Y^6-6 Y^5+12 Y^4+14 Y^3+2 Y^2 & 2 Y^7+20 Y^5-48 Y^4-6 Y^3-Y^2+6 Y & 8 Y^4-Y^5 \\
 Y^3-Y & 6 Y^2-Y^3 & -Y^2 \\
\end{psmallmatrix}
\]

Application of Proposition \ref{pro:3x3} gives the following annihilator:
\begin{multline*}
-(Y-1)^2 \left(Y^4-4 Y^3-Y^2-4 Y+1\right)^2 \left(Y^4+3 Y^3+6 Y^2+3 Y+1\right) \\ \left(2 Y^7+20 Y^5-48 Y^4-6 Y^3-Y^2+6 Y-1\right).
\end{multline*}
Using Proposition \ref{pro:minimalpolynomial} it is then straightforward to find the minimal annihilators for the numerator and denominator, which are 
\begin{equation}\label{equ:annihilator14}
(-1 + X)^2 (1 - 4 X - X^2 - 4 X^3 + X^4)^2 (1 + 3 X + 6 X^2 + 3 X^3 + X^4) \end{equation}
 and
\begin{equation}\label{equ: annihilator5}(X-1) (1 - 4 X - X^2 - 4 X^3 + X^4) 
\end{equation}
respectively.
Because the factors of both the numerator and denominator are of degree 4 or lower we may obtain the roots in explicit algebraic form. For the denominator the roots, all of multiplicity one, are 1 and $\frac{1}{2} \left(2+\sqrt{7} \pm \sqrt{4 \sqrt{7}+7}\right)$ and
$\frac{1}{2} \left(2-\sqrt{7} \pm i \sqrt{4 \sqrt{7}-7}\right).$ The numerator has these roots with multiplicity two and additionally has roots (of multiplicity one) $\frac{i \sqrt{7}}{4} \pm \frac{1}{2} \sqrt{-\frac{7}{2}-\frac{1}{2} 3 i \sqrt{7}}-\frac{3}{4}$ and $
-\frac{i \sqrt{7}}{4} \pm \frac{1}{2} \sqrt{-\frac{7}{2}+\frac{3 i \sqrt{7}}{2}}-\frac{3}{4}.$ The
recursions for the numerator and denominator are valid for determinants of size $n \ge 18$ and $n \ge 10$  respectively.  Consequently, \eqref{equ:Bapat} is valid for graphs of size 20 or greater. 
 Using \eqref{equ:Bapat} we can obtain closed Binet formulas for the numerator and denominator and consequently can obtain an explicit rational function for the resistance distance. The numerators and denominators of this rational function are a linear sum of powers of nine algebraic numbers; using Mathematica 13.3 we can then simplify the difference between successive resistances proving the one over 14 conjecture.

 Because the formula is only valid for $n \ge 20,$ the coefficients in the closed Binet form are quite complicated.
Due to this complexity, we suffice with giving the explicit forms for the asymptotic expressions for the numerator and denominator. The one over 14 conjecture does not a priori apply to asymptotic formula, only to exact formula. Nevertheless, in this case, the differences between successive terms of the natural asymptotic formula for resistance  (based on using the dominant root) are also exactly one over 14.
 
Fortunately, although a priori there is no reason why this should be true, we found that if we  start the sequences at indices 8 and 11, for the numerator and denominator respectively,  remarkably the resulting asymptotic formula besides satisfying the one over 14 conjecture  are also quite good approximations to the actual sequence. We provide details below for the asymptotic formula and simply note that the exact case is treated similarly and hence omitted.

\subsection{Proof of the one over 14 conjecture for the asymptotic formula} We use the method of generating functions \cite{Wilf}. To find the exact formula we need to calculate all coefficients in the Binet form and that requires solving a system of equations. For the asymptotic formula we only need two coefficients corresponding to the dominant root and that can be done without solving for the other coefficients. To use \eqref{equ:Bapat} we need to find separate asymptotic formulas for the numerator and denominator. 

\textbf{The Numerator.} The theory of generating functions naturally works with sequences whose initial index is 0. Accordingly, define 
\begin{equation}\label{equ:hngnnumerator}
h_n = L(\{1,n+2\}|\{1,n+2\}), \qquad  g_{i} =h_{8+i}, \qquad i \ge 0.\end{equation}
where $L$ is the Laplacian for the linear 3--tree on $n+2$ nodes. Thus, 
\begin{multline*}
G=\{ g_0,g_1, \dotsc\}=\\
\{ 127920, 606530, 2858661, 13426688, 62846424, 293216196, 1364289416, \\
6331841700, 29319607080, 135483247712, 624865625995, \dotsc \}.
\end{multline*}

To clarify our use of indices, by \eqref{equ:hngnnumerator}, $g_0 = 127,920$ which implies
that $h_8 = 127,920$ with $h_8 = L(\{1,10\}|\{1,10\}).$ We will derive the asymptotic formula for $G$ then shift back 8 and finally use $n-2$ in the numerator of~\eqref{equ:Bapat} when calculating resistance distance.

This sequence $G$ satisfies the recursion corresponding to the annihilator given by \eqref{equ:annihilator14}  with dominant root of multiplicity 2, given by
$r =\frac{1}{2} \left(2+\sqrt{7} + \sqrt{4 \sqrt{7}+7}\right) \approx 4.41948.$   
The generating function for the sequence $G,$ is
$$
GF_{numerator}= \frac{N_{num}}{D_{num}},
$$
with 
\begin{align*}
    N_{num}=127920 - 288910 X - 491609 X^2 - 2338229 X^3 + 1406395 X^4 - 
 445536 X^5 \\ 
 + 5047630 X^6 - 1302912 X^7 + 951712 X^8 - 2640674 X^9 \\ - 
 31621 X^{10} - 149867 X^{11} + 182525 X^{12} - 26880 X^{13},
\end{align*}
and 
$$    D_{num} = \prod_{i=1}^9\left(\frac{1}{1- r_i X} \right)^{m_i}, $$
with the $r_i$ and $m_i$ the distinct roots (and their multiplicities) of \eqref{equ:annihilator14} set equal to 0.

As mentioned above, since we are interested in the asymptotic formula we only care about the dominant root. Towards this end we use a partial fraction decomposition
$$
GF_{numerator}= \frac{C_{num,1}}{1-r_1 X}+
\frac{C_{num,2}}{(1-r_1 X)^2}.
+
 \sum_{i=2}^9 \sum_{j_i=1}^{m_i}\frac{C_{num,i}}{((1-r_i X)^{j_i}},
$$
with $$C_{num,1}, C_{num,2} $$ constants. We can obtain $C_{num,2}$ by multiplying both sides of the partial fraction decomposition by $(1-r_1 X)^2,$ plugging into the resulting left-hand side $X= \frac{1}{r_1}$, taking limits and evaluating. We similarly can obtain $C_{num,1}$ by multiplying both sides of  the partial fraction decomposition by $(1-r_1 X)^2,$ then differentiating both sides of the resulting equation by $X,$ dividing by $-r_1,$ plugging into the resulting left-hand side $\frac{1}{r_1},$ 
taking limits and evaluating.

It follows that the asymptotic approximation to the sequence $G$ is given by
$$
    g_n \approx C_{num,1} r_1^n+C_{num,2}(n+1) r_1^n,
$$
and therefore by \eqref{equ:hngnnumerator}, 
$$
    h_n \approx C'_{num,1} r_1^n+C'_{num,2}(n-7) r_1^n
$$
with 
\begin{align*}
    C'_{num,2} &= \frac{256 \left(47540907929 \sqrt{7}+29996455428 \sqrt{4 \sqrt{7}+7}+11337594468 \sqrt{7 \left(4 \sqrt{7}+7\right)}+125781419483\right)}{49 \left(\sqrt{7}+\sqrt{4 \sqrt{7}+7}+2\right)^9 \left(11955 \sqrt{7}+7543 \sqrt{4 \sqrt{7}+7}+2851 \sqrt{7 \left(4 \sqrt{7}+7\right)}+31629\right)}
    \\ &\approx 0.0630896
\end{align*}
and 
\scriptsize
\begin{align*} 
    C'_{num,1} &=\frac{24576 \left(13278615154497729 \sqrt{7}+8378287352473094 \sqrt{4 \sqrt{7}+7}+3166694963897366 \sqrt{7 \left(4 \sqrt{7}+7\right)}+35131913454134511\right)}{7 \left(\sqrt{7}+\sqrt{4 \sqrt{7}+7}\right)^3 \left(\sqrt{7}+\sqrt{4 \sqrt{7}+7}+2\right)^9 \left(7157 \sqrt{7}+4515 \sqrt{4 \sqrt{7}+7}+1707 \sqrt{7 \left(4 \sqrt{7}+7\right)}+18935\right)^2}\\  &\approx 0.816459.
\end{align*}
\normalsize

The ratios of this asymptotic formula over the exact values of $h_n, 8 \le n \le 15,$ are
$$
\{1.00067, 0.999617, 1.00007, 1.00004, 0.999965, 1.00001, 1., \
0.999997, 1., 1., 1., 1., 1.\}
$$
showing good agreement.

\subsection*{Denominator}  As with the numerator we define sequences $g_i,h_i$ by
\begin{equation}\label{equ:giluvd} 
h_n = L(\{1,n+1\}|\{1,n+1\}), \qquad g_{i} =h_{11+i} , i \ge 0. \end{equation}
where $L$ is the Laplacian for the linear 3--tree on $n+1$ nodes and
we  have shifted by 11.
The sequence $G = \{g_i\}_{i \ge 0}$ is annihilated by \eqref{equ: annihilator5}.

The generating function for the sequence $G,$ is
$$
GF_{denominator}= \frac{N_{den}}{D_{den}}
$$
with 
$$
    N_{dem}=568101 X^4-2711961 X^3+1090668 X^2-1457516 X+2510716,
$$
and 
$$
    D_{den} = \prod_{i=1}^5 \frac{1}{1- r_i X}.
$$
As with the numerator, using a partial fraction decomposition
$$
GF_{denominator}= \frac{C_{den}}{1-r_1 X}+
\sum_{i=2}^5 \frac{C_{den,i}}{1-r_i X},
$$
with $C_{den}$ a constant. 
It follows that the asymptotic approximation to the sequence $G$ is given by
$$
    g_n \approx C_{den} r_1^n,
$$
and therefore by \eqref{equ:giluvd}
$$
   h_n \approx C'_{den} r_1^n, 
$$
with  
\begin{align*}
    C'_{den} &= \frac{1024 \left(159102007 \sqrt{7}+100387151 \sqrt{4 \sqrt{7}+7}+37942776 \sqrt{7 \left(4 \sqrt{7}+7\right)}+420944344\right)}{7 \left(\sqrt{7}+\sqrt{4 \sqrt{7}+7}+2\right)^{11} \left(5 \sqrt{7}+3 \sqrt{4 \sqrt{7}+7}+\sqrt{7 \left(4 \sqrt{7}+7\right)}+11\right)} \\ &\approx 0.199855.
\end{align*}

The approximation is quite good even for $n<11.$ For 
 $6 \le n \le  15$ the ratio of the asymptotic formula over the exact value, $h_n$ is given by
 $$ \{1.00078, 1.0002, 1.00003, 1.00001, 1., 1., 1., 1., 1., 1.\}.$$

We can now prove the one over 14 conjecture for the asymptotic formula.
 \begin{theorem}\label{the:oneover14}
Define
$$R(1,n) = 
\frac{C_{num,2} (n-9) r_1^{n-2} + C_{num,1} r_1^{n-2}}
{C_{den} r_1^{n-1}}
$$
Then 
$$R(1,n+1) - R(n) =  \frac{1}{14}.$$
 \end{theorem}
\begin{proof} We use the asymptotic formula obtained above with the explicit values of constants. Algebraically simplifying the closed formula (using  Mathematica 13.3) we obtain the desired  result. \end{proof}
 
\section{Conclusion and Open Problems}

This paper has generalized and semi-automated one approach to calculating resistance distances  by first finding the underlying recursions associated with the determinants of matrices related to the underlying families of graphs. We have already indicated several outstanding problems in earlier sections of the paper. In this closing section, we speculate on further uses of this approach and list some additional open problems.

For graph families such as the path, fan, wheel, and ladder graphs both the approach of this paper and the more traditional approaches, discovering closed formula and then proving them using inductive arguments are equally effective. However, when graph families become complex, the LaplaceExpand approach is more efficient. This was illustrated with the linear 3--tree. We feel that future uses of the LaplaceExpand approach could fruitfully be applied to more complicated graph families. We illustrate with the corrugated linear 2--tree which is a linear 2--tree on $n$ vertices where $n=3m+2$, with $2m-2$ bends located at vertices $4, 3m$ and vertices $3k+3, 3k+4$ for $k\in\{1,2,\ldots, m-1\}$.  The corrugated linear 2-tree on 17 vertices is displayed in Figure~\ref{fig:caterpillar}. 

\begin{figure}
\begin{center}
\begin{tikzpicture}[line cap=round,line join=round,>=triangle 45,x=1.0cm,y=1.0cm,scale = 1.2]
\draw [line width=1.pt] (-3.,0.)-- (-2.,0.);
\draw [line width=1.pt] (-2.,0.)-- (-1.,0.);
\draw [line width=1.pt] (-1.,0.)-- (0.,0.);
\draw [line width=1.pt] (0.,0.)-- (1.,0.);
\draw [line width=1.pt] (1.,0.)-- (2.,0.);
\draw [line width=1.pt] (3.,0.)-- (2.,0.);
\draw [line width=1.pt] (2.,0.)-- (1.5,0.866025403784435);
\draw [line width=1.pt] (1.5,0.866025403784435)-- (1.,0.);
\draw [line width=1.pt] (1.,0.)-- (0.5,0.8660254037844366);
\draw [line width=1.pt] (0.5,0.8660254037844366)-- (0.,0.);
\draw [line width=1.pt] (0.,0.)-- (-0.5,0.8660254037844378);
\draw [line width=1.pt] (-0.5,0.8660254037844378)-- (-1.,0.);
\draw [line width=1.pt] (-1.,0.)-- (-1.5,0.8660254037844385);
\draw [line width=1.pt] (-1.5,0.8660254037844385)-- (-2.,0.);
\draw [line width=1.pt] (-2.,0.)-- (-2.5,0.8660254037844388);
\draw [line width=1.pt] (-2.5,0.8660254037844388)-- (-3.,0.);
\draw [line width=1.pt] (-1.5,-0.8660254037844388)-- (-2.,0.);
\draw [line width=1.pt] (-1.5,-0.8660254037844388)-- (-1.,0.);
\draw [line width=1.pt] (-.5,-0.8660254037844388)-- (0.,0.);
\draw [line width=1.pt] (-.5,-0.8660254037844388)-- (-1.,0.);
\draw [line width=1.pt] (-2.5,0.8660254037844388)-- (-1.5,0.8660254037844385);
\draw [line width=1.pt] (-1.5,-0.8660254037844385)-- (-0.5,-0.8660254037844378);
\draw [line width=1.pt] (-0.5,0.8660254037844378)-- (0.5,0.8660254037844366);
\draw [line width=1.pt] (1.5,0.8660254037844378)-- (2.5,0.8660254037844366);
\draw [line width=1.pt] (0.5,-0.8660254037844366)-- (1.5,-0.866025403784435);
\draw [line width=1.pt] (0.5,-0.8660254037844366)-- (0.0,0.0);
\draw [line width=1.pt] (0.5,-0.8660254037844366)-- (1.0,0.0);
\draw [line width=1.pt] (1.5,-0.8660254037844366)-- (1.0,0.0);
\draw [line width=1.pt] (1.5,-0.8660254037844366)-- (2.0,0.0);
\draw [line width=1.pt] (2.5,0.8660254037844366)-- (2.0,0.0);
\draw [line width=1.pt] (2.5,0.8660254037844366)-- (3.0,0.0);
\begin{scriptsize}
\draw [fill=black] (-3.,0.) circle (1.5pt);
\draw[color=black] (-3,-0.3) node {$1$};
\draw [fill=black] (-2.,0.) circle (1.5pt);
\draw[color=black] (-2,-0.3) node {$3$};
\draw [fill=black] (-2.5,0.867) circle (1.5pt);
\draw[color=black] (-2.5,1.1) node {$2$};
\draw [fill=black] (-1.5,0.867) circle (1.5pt);
\draw [fill=black] (-1.5,-0.867) circle (1.5pt);
\draw[color=black] (-1.5,1.1) node {$4$};
\draw [fill=black] (-1.,0.) circle (1.5pt);
\draw[color=black] (-1,-0.3) node {$5$};
\draw [fill=black] (-0.5,0.8660254037844378) circle (1.5pt);
\draw [fill=black] (-0.5,-0.8660254037844378) circle (1.5pt);
\draw[color=black] (-1.5,-1.1) node {$6$};
\draw[color=black] (-.5,-1.1) node {$7$};
\draw [fill=black] (0.,0.) circle (1.5pt);
\draw[color=black] (0,-0.3) node {$8$};
\draw [fill=black] (0.5,0.867) circle (1.5pt);
\draw [fill=black] (1.5,0.867) circle (1.5pt);
\draw [fill=black] (0.5,-0.867) circle (1.5pt);
\draw[color=black] (-0.5,1.1) node {$9$};
\draw[color=black] (0.5,1.1) node {$10$};
\draw[color=black] (0.5,-1.1) node {$12$};
\draw [fill=black] (1.,0.) circle (1.5pt);
\draw[color=black] (1,-0.3) node {$11$};
\draw[color=black] (1.5,-1.1) node {$13$};
\draw[color=black] (2,-0.3) node {$14$};
\draw [fill=black] (1.5,0.87) circle (1.5pt);
\draw [fill=black] (2.5,0.87) circle (1.5pt);
\draw [fill=black] (1.5,-0.87) circle (1.5pt);
\draw[color=black] (1.5,1.1) node {$15$};
\draw[color=black] (2.5,1.1) node {$16$};
\draw [fill=black] (2.,0.) circle (1.5pt);
\draw [fill=black] (3.,0.) circle (1.5pt);
\draw[color=black] (3,-0.3) node {$17$};
\end{scriptsize}
\end{tikzpicture}
\end{center}
    \caption{A corrugated linear 2-tree on 17 vertices }
    \label{fig:caterpillar}
\end{figure}
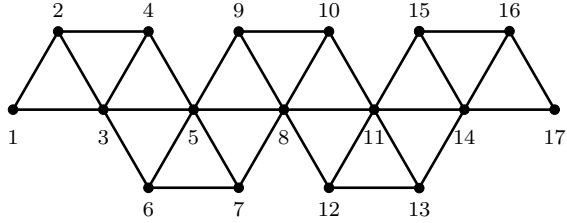
In~\cite{swim2019} the first author determined both a closed formula for the resistance distance between the degree-two vertices as well as an asymptotic limit as the number of vertices grows.  These results were obtained by starting with a straight linear 2--tree and determining the impact of each bend. It is natural to see what the LaplaceExpand approach of this paper yields.  This approach requirs 834 Laplace Reductions and introduces 423 matrices. The reduced system has 26 determinant identities in 26 determinant families. 

\begin{op} Theorem~\ref{them:allfam} assures us that an $n \times n$ system can be reduced to a single annihilator of one determinant family. Find a closed formula using the coefficient entries in the associated $n \times n$ matrix which immediately gives the annihilator. \end{op}

We now state another open problem. As just indicated, the corrugated linear tree requires 834 Laplace expansions and introduces 423 matrix families. For the linear 3--tree, 201 Laplace expansions are needed producing 80 matrix families. However, prior to automating the approach, the authors approached this problem manually. This enabled quick visual inspection of the matrix families introduced after each Laplace expansion, and with a judicious choice of expansions, required only 21 matrix families producing a roughly 75\% savings over the 80 families produced by the automated LaplaceExpand. Heuristically, we selected for reduction, matrix families that occurred more frequently to that point. 

\begin{op} Find a method to reduce the number of expansions and the number of matrix families introduced.\end{op}

\def\cprime{$'$}

 \section{Appendix I - Computer programs}

This Appendix contains the programs, written in Mathematica 13.3, LaplaceExpand and SystemReduce described in Section~\ref{sec:algorithm}. Some supportive functions needed by these programs are also listed.

\vspace{2mm}

\textbf{SystemReduce}

\begin{verbatim}
  
SystemReduce[LaplaceExpansionMatrix_, 
   LaplaceExpansionIncidenceMatrix_] :=
  Module[{M, k, i, j, P}, 
  M = LaplaceExpansionMatrix; 
  P =LaplaceExpansionIncidenceMatrix; 
   
  For[k = 2, k <= Dimensions[P][[1]], k++,
     If[P[[k, k]] === 0,
            M = 
                Table[If[P[[i, k]] == 1, 
                If[i != k, 
                If[j != k, 
                M[[i, j]] + M[[i, k]] M[[k, j]], 0], 
                M[[i, j]]], 
                M[[i, j]]],
                {i, 1, Dimensions[P][[1]]},
                {j, 1, Dimensions[P][[2]]}];
            P = 
                Table[If[P[[i, k]] == 1, 
                If[i != k, 
                If[j != k, 
                If[P[[i, j]] === 0 && P[[k, j]] === 0, 0, 1], 0], 
                P[[i, j]]], P[[i, j]]], 
                {i, 1, Dimensions[P][[1]]}, 
                {j, 1,  Dimensions[P][[2]]}];
      
     Print["This is the reduction corresponding to row ", k];
     Print[M // MatrixForm]]]];

DropM[M1_, r1_, c1_] := 
  Module[{M, r, c, b}, 
  M = M1; r = r1; c = c1; 
  b = Delete[M, r]; 
  Return[Transpose[Delete[Transpose[b], c]]]];

 
AddRow[x_] := Append[x, ConstantArray[0, Dimensions[x][[2]]]];

AddCol[x_] := 
  Transpose[
   Append[
    Transpose[x], ConstantArray[0, Dimensions[x][[1]]]]];
 \end{verbatim}

 \textbf{LaplaceExpand}

 \begin{verbatim}
 
 LaplaceExpand[MatrixFamily_, MinimumSizeForVerification_] := 
 Module[ {a, i, j, FlagFound, FlagFoundTemp, jFlagFound, 
        P, Q, R, M, M2, RC, 
        RowCount, IDCount, ParentCount, LoopCount, MTemp, MinSiz},
  
  a[n_] := MatrixFamily[n]; MinSiz = MinimumSizeForVerification;
  
  M = Table[0, {i, 1, 1}, {j, 1, 7}];
  RowCount = 1; ParentCount = 1; IDCount = 1; 
  Q = Association[1 -> 1]; 
  P = Association[1 -> 1]; 
  R = Association[1 -> 1];
  M2[1, n_] := a[n];
  
  M[[RowCount, 1]] = IDCount;
  M[[RowCount, 2]] = 1;  
  RC[n_] := RC[n] = If[
        Count[M2[n, MinSiz][[All, 1]], Except[0]] < 
        Count[M2[n, MinSiz][[1]], Except[0]], "C", "R"];
  
        M[[RowCount, 3]] = RC[1]; 


  LoopCount = 0;
  Until[Total[M[[All, 2]] ] == 0, 
    LoopCount++;
   
    For[i = 1, i <= MinSiz, i++,                                            
    
    If[RC[ParentCount] == "R" && 
        M2[ParentCount, MinSiz][[1, i]] == 0  || 
        RC[ParentCount] == "C" && 
        M2[ParentCount, MinSiz][[i, 1]] == 0 ,  ,   
        RowCount++; M = AddRow[M];  
        
        Clear[MTemp];     
        MTemp[n_] := MTemp[n] =
            If[RC[ParentCount] == "R",
                Drop[M2[ParentCount, n + 1], 1, i],  
                Drop[M2[ParentCount, n + 1], {i}, {1}]]; 
      			 
      FlagFound = False; j = 1; 
      Until[FlagFound == True || j > ParentCount, 
            
        FlagFoundTemp = True;
        For[l = MinSiz, l <= MinSiz + 3, l++, 
            FlagFoundTemp = 
                FlagFoundTemp &&
                (MTemp[l] == M2[j, l] || 
                MTemp[l] == Transpose[M2[j,l]])];
            FlagFound = FlagFoundTemp;
            jFlagFound = If[FlagFound, j, 0];
        j++]; 
           
    If[FlagFound == True, 
        M[[RowCount, 1]] = jFlagFound;
        M[[RowCount, 4]] = ParentCount;
        If[RC[ParentCount] == "R",
            M[[RowCount, 5]] = 1;     
            M[[RowCount, 6]] = i,
        	    
            M[[RowCount, 5]] = i; 
            M[[RowCount, 6]] = 1]; 
            M[[RowCount, 7]] = 
                (-1)^(i + 1) *
                Y *
                If[RC[ParentCount] == "R", 
                    M2[ParentCount, MinSiz][[1, i]], 
                    M2[ParentCount, MinSiz][[i, 1]]  ], 
            IDCount++; 
            M[[RowCount, 1]] = IDCount; 
                AssociateTo[Q, IDCount -> i];
                AssociateTo[P, IDCount -> ParentCount]; 
                AssociateTo[R, IDCount -> RowCount];
            M2[IDCount_, n_] := M2[IDCount, n] =
            If[RC[P[IDCount]] == "R",
                Drop[M2[P[IDCount], n + 1], 
                {1}, 
                {Q[IDCount]}],
          		                            
                Drop[M2[P[IDCount], n + 1], 
                    {Q[IDCount]}, 
                    {1}]];   
            M[[RowCount, 2]] = 1; 
            M[[RowCount, 3]] = RC[IDCount]; 
            M[[RowCount, 4]] = ParentCount;
            If[RC[P[IDCount]] == "R",
                M[[RowCount, 5]] = 1; M[[RowCount, 6]] = i,
                M[[RowCount, 5]] = i; M[[RowCount, 6]] = 1];
            M[[RowCount, 7]] = 
                (-1)^(i + 1) *
                Y * 
                If[RC[P[IDCount]] == "R",
                    M2[P[IDCount], MinSiz][[1, i]], 
                    M2[P[IDCount], MinSiz][[i, 1]]  ]; 
    ];  
         
    ];   
    ]; 
    
   M[[R[[Key[ParentCount]]], 2]] = 0;
    
   If[Total[M[[All, 2]] ] == 0, , 
    ParentCount = M[[Position[M[[All, 2]], 1][[1]][[1]],1]]]; 
    
     ]; 
  
   
  Print["This matrix gives the sequence of Laplace Expansions ", 
   M // MatrixForm];
  For[i = 1, i <= IDCount, i++, 
   Print["This is matrix family with IDCount ", 
        i, "   ", 
        M2[i, MinSiz] // MatrixForm]];
   
  Clear[Q, P]; 
  Q = Table[0, {i, 1, IDCount}, {j, 1, IDCount}];
  
  For[i = 2, i <= Dimensions[M][[1]], i++, 
   Q[[M[[i, 4]], M[[i, 1]]]] = M[[i, 7]]]; 
  
  Print["This is the system of matrix family identities ", 
   Q // MatrixForm];
  
  P = Table[0, 
        {i, 1, Dimensions[Q][[1]]}, 
        {j, 1, Dimensions[Q][[2]]}];
  P = Table[
        If[Q[[i, j]] === 0, 0, 1], 
        {i, 1, Dimensions[P][[1]]}, 
        {j, 1,  Dimensions[P][[2]]}];  
  SystemReduce[Q, P]
  ]
\end{verbatim}

\end{document}